\documentclass{amsart}
\usepackage{amssymb}
\usepackage{amsmath}
\usepackage{amsfonts}
\usepackage[dvips]{graphicx}
\newtheorem{theorem}{Theorem}[section]

\newtheorem{lemma}[theorem]{Lemma}

\newtheorem{proposition}[theorem]{Proposition}

\newtheorem*{lem}{Lemma}

\renewcommand{\Im}{\mathop{\mathrm{Im}}\nolimits}
\renewcommand{\Re}{\mathop{\mathrm{Re}}\nolimits}

\theoremstyle{definition}
\newtheorem{definition}[theorem]{Definition}

\theoremstyle{remark}
\newtheorem{remark}[theorem]{Remark}

\numberwithin{equation}{section}

\DeclareMathOperator{\supp}{supp}

\DeclareMathOperator{\id}{id}
\DeclareMathOperator{\sgn}{sgn}

\begin{document}

\title[Spaces of smooth functions]{Differential expressions with mixed
homogeneity and spaces of smooth functions they generate}
\author{S. V. Kislyakov, D. V. Maksimov, and D. M. Stolyarov}

\thanks{Supported by RFBR, grant no. 11-01-00526, and by the Chebychev Research Laboratory, a grant
by the Government of Russia, contract 11.G34.31.0026}

\address{St. Petersburg Department of the 
V. A. Steklov Math. Institute \\
27 Fontanka, St. Petersburg \\
191023, Russia}

\email{skis@pdmi.ras.ru}

\address{St. Petersburg Polytechnical State University \\
29, Polytechnicheskaya st., St. Petersburg \\
195251, Russia}
 
\email{dimax239@bk.ru}

\address{P. L. Chebyshev Research laboratory \\
St. Petersburg State University \\
29b, 14th Line of Vasil'evskii Island and 
St. Petersburg Department of the 
V. A. Steklov Math. Institute \\
27 Fontanka, St. Petersburg \\
191023, Russia}

\email{dms239@mail.ru}

\begin{abstract}
Let $\{T_1,\dots,T_l\}$ be a collection of differential operators with constant coefficients on the
torus $\mathbb{T}^n$. Consider the Banach space $X$ of functions $f$ on the torus
for which all functions $T_j f$, $j=1,\dots,l$, are continuous. Extending the previous work of
the first two authors, we analyse the embeddability of $X$ into some space $C(K)$
as a complemented subspace. We prove the following. Fix some pattern of mixed homogeneity and
extract the senior homogeneous parts (relative to the pattern chosen) $\{\tau_1,\dots,\tau_l\}$
from the initial operators $\{T_1,\dots,T_l\}$.
Let $N$ be the dimension of the linear span of $\{\tau_1,\dots,\tau_l\}$. If $N\geqslant 2$,
then $X$ is not isomorphic to a complemented subspace of $C(K)$ for any compact space $K$.

The main ingredient of the proof of this fact is a new Sobolev-type embedding theorem.

\end{abstract}
\maketitle

\section{Introduction}

The space $C^{(k)}(\mathbb{T})$ of $k$ times continuously differentiable functions on the unit
circle $\mathbb{T}$ of the complex plane is isomorphic to $C(\mathbb{T})$ (modulo constants,
the isomorphism is given by $k$-fold differentiation). However, it has long been known
that, already for the $2$-dimensional torus $\mathbb{T}^2$, the situation is different. The
understanding of this phenomenon has been increasing gradually, starting with \cite{G} and \cite{H},
and then through the work done in \cite{K1}, \cite{K2}, \cite{KwP}, \cite{Si}, \cite{PS}, \cite{KSi},
\cite{KM1}, \cite{KM2}, and \cite{M}.

We begin our discussion directly with the general framework considered
in \cite{KM1}. Let $T=(T_1,\dots,T_l)$ be
a collection of differential operators with constant coefficients on the torus $\mathbb{T}^n$. This means that
each $T_j$ is a linear combination of differential monomials
$D^{\alpha}=\partial_1^{\alpha_1}\dots\partial_n^{\alpha_n}$. Here
$\alpha=\{\alpha_1,\dots,\alpha_n\}$ is a multiindex composed of nonnegative integers. Next,
by $\partial_j f$ we mean the operator
$g\mapsto\frac{\partial}{\partial t}g(e^{2\pi it})$, $t\in\mathbb{R}$, applied to $f$ with
respect to the $j$th variable.\footnote{In general, we use the mapping
$t\mapsto e^{2\pi it}$ to parametrise the unit circle; so, the
trigonometric system is $\{e^{2\pi ikt}\}_{k\in\mathbb{Z}},\,\, t\in[0, 1)$, and
the (multiple) Fourier series are handled accordingly.}
The quantity $|\alpha|=\alpha_1+\dots +\alpha_n$ is
called the \textit{order} of the differential monomial. The order of a differential
operator is the largest order of a differential monomial involved in the operator.

The above collection $T$ gives rise to the following seminorm on the set of trigonometric
polynomials in $n$ variables:
\[
\|f\|_{T}=\max_{1\leqslant j\leqslant l}\|T_j f\|_{C(\mathbb{T}^n)}.
\]
The Banach space determined by this seminorm (i.e., the result of factorization over the
null-space and completion) is denoted by $C^T(\mathbb{T}^n)$ and is called the space of smooth
function generated by the collection $T$.

When $T$ consists of all differential \textit{monomials} of order at most $s$, we obtain the classical
space $C^{(s)}(\mathbb{T}^n)$ of $s$ times continuously differentiable functions on the
$n$-dimensional torus. It is known that if $n \geqslant 2$ and $s \geqslant 1$, then the bidual of this space
does not embed in a Banach lattice as a complemented subspace, see \cite{K2}, \cite{KwP}. In
particular, this bidual (\textit{a fortiori}, the space itself)
is not isomorphic to a complemented subspace of $C(K)$ for any
compact space $K$.

If $T$ is an \textit{arbitrary} finite collection of differential monomials, we deal with
the classical \textit{anisotropic} spaces of smooth functions. The isomorphism problem for them
was treated in \cite{PS}, \cite{Si} and \cite{KSi}. Not giving a precise statement, we
signalize that the following dichotomy occurs: if the collection
$T$ contains a senior differential monomial (i.e., the monomial whose multiindex dominates
coordinatewise all other multiindices involved), then, up to some not quite essential
subtleties (see
\cite{KSi}), $C^T(\mathbb{T}^n)$ is isomorphic to $C(\mathbb{T}^n)$; otherwise, again,
the bidual of $C^T(\mathbb{T}^n)$ is not embeddable complementedly in a $C(K)$-space (more
generally, in a Banach lattice).

The importance of the absence of a ``senior'' operator for nonisomorphism
was further emphasized by the results of \cite{KM1}, \cite{KM2}. In those papers, the case
of a collection $T$ consisting not necessarily of differential monomials was treated for the first
time. 

The main result of \cite{KM1}, \cite{KM2} says the following. Suppose all operators in the
collection $T$ are of order not exceeding $s>0$. In every operator $T_j$ of the collection, we
drop its \textit{junior part}, i.e., all differential monomials of order strictly smaller than
$s$. The remaining \textit{senior part} $\tau_j$ is a homogeneous differential operator of
order precisely $s$. \textit{If there are two linearly independent operators among the $\tau_j$,
$j=1,\dots,l$, then the bidual of $C^T(\mathbb{T}^n)$ is not isomorphic to a complemented subspace
of a $C(K)$-space.}

But if all $\tau_j$ are multiples of one of them, the situation was still unclear. More precisely,
in the case of two-dimensional torus, in \cite{KM1} it was shown that then the space
$C^T(\mathbb{T}^2)$ is isomorphic indeed to a
$C(K)$-space if the junior parts of all $T_j$'s vanish. (In higher dimension the picture is
more complicated.) However, if they do not, nonisomorphism may
again occur. We already saw this when we discussed
anisotropic spaces of smooth functions: two incomparable maximal\footnote{The
terms ``incomparable" and ``maximal" are related to the coordinatewise partial ordering
of multiindices.} monomials involved in the
definition of such a space need not be of one and the same order $s$, though the space itself
is definitely not of type $C(K)$ if two such monomials exist.

This suggests that the concept of mixed homogeneity (permeating the theory of anisotropic spaces)
may play a role also in the general situation. The main result of this paper says that
it is indeed the case. This can be viewed as a joint refinement of the results of \cite{PS}, \cite{Si},
\cite{KSi} and \cite{KM1}, \cite{KM2}.

Geometrically, a \textit{mixed homogeneity pattern} in $n$
variables is determined by a hyperplane $\Lambda$ intersecting the $n$ positive coordinate semiaxes.
The equation of such a hyperplane is $\sum_{k=1}^n\frac{x_k}{a_k}=1$,
where the $a_k$ are positive numbers. A differential operator $S$ is said to be $\Lambda$-homogeneous
if all points corresponding to the multiindices of differential monomials involved
in $S$ belong to $\Lambda$. 
Consider a finite collection $T=\{T_1,\dots,T_l\}$ of differential
operators in $n$ variables such that
\begin{equation}\label{odn}
\sum_{k=1}^n\frac{\alpha_k}{a_k}\leqslant 1
\end{equation}
for every differential monomial $\partial_1^{\alpha_1}\dots\partial_n^{\alpha_n}$
involved in at least one of $T_j$. Dropping all terms in $T_j$ for which we have
strict inequality in (\ref{odn}), we obtain a differential operator $\tau_j$, to
be called the $\Lambda$-senior part of $T_j$.

\begin{theorem}\label{main1}
If for at least one choice of $\Lambda$ there are at least two linearly independent
operators in the collection of $\Lambda$-senior parts, then the
bidual of $C^T(\mathbb{T}^n)$ is not isomorphic to a complemented subspace
of a $C(K)$-space.
\end{theorem}

We want to say immediately that this theorem reduces easily to the case of the two-dimensional
torus. This requires slight concentration, but similar procedures were described in \cite{KSi},
\cite{KM1}, and \cite{PS}, so we only give a hint to the proof. The space $C^T(\mathbb{T}^n)$ is
determined in fact by the linear span of the collection $T$; so, if this collection gives rise to
two linearly independed $\Lambda$-senior parts, there is no loss of generality in assuming
the existence of two nonequal differential monomials, each occurring in precisely one
of these $\Lambda$-senior parts. Now, a guideline for the reduction is
to draw the $2$-plane through the origin and the points corresponding to the multiindices of
these two monomials. Consult the papers cited above for more details.

So, in what follows we mainly restrict ourselves to the case of $n=2$. It should be noted that,
in all the above references except for the early papers \cite{G} and \cite{H},
nonisomorphism was proved
by combining the techniques of absolutely summing operators with a Sobolev-type embedding
theorem. Here we do the same, more specifically, we try to imitate the pattern of \cite{KM1},
\cite{KM2}. However, the required embedding theorem is not quite standard.
To present the statement, we remind the reader the definition of a Sobolev
space with nonintegral smoothness:
\begin{equation}\label{sobolev}
W_2^{\alpha,\beta}(\mathbb{T}^2)=
\{f\in C^{\infty}(\mathbb{T}^2)^{\prime} \colon
\{(1 + m^2)^{\frac{\alpha}{2}} (1 + n^2)^{\frac{\beta}{2}}\hat{f}(m,n)\}
\in l^2(\mathbb{Z}^2)\}.
\end{equation}
We need yet another notion.
\begin{definition}\label{PROPER}
A distribution $f$ on $\mathbb{T}^2$ is said to be proper if
$\hat{f}(s,t)=0$ whenever $s=0$ or $t=0$, $s,t\in\mathbb{Z}$.
\end{definition}
In particular, we can talk about a proper measure, or (integrable) function, or
trigonometric polynomial. Note also that, on the set of proper functions, the norm
of the space \eqref{sobolev} is equivalent to the norm
$\left\|f\right\|_{W} = \left\| \{|m|^{\alpha}|n|^{\beta}\hat{f}(m,n)\}\right\|_{l^2(\mathbb{Z}^2)}$. 
\begin{theorem}\label{main2}
Suppose that proper distributions
$\varphi_1,\dots,\varphi_N$ on the $2$-dimensional torus satisfy the
system of equations
\begin{equation} \label{syst}
-\partial_1^k \varphi_1 = \mu_0;\quad
\partial_2^l \varphi_{j} - \partial_1^k \varphi_{j+1} = \mu_{j},
\quad j= 1,\dots, N-1;\quad \partial_2^l \varphi_N = \mu_{N},
\end{equation}
where $\mu_0,\dots,\mu_N$ are measures. If either $k$ or $l$ is odd, then
\begin{equation}\label{est1}
\sum\limits_{j = 1}^{N}\|\varphi_j\|_{W_2^{\frac{k-1}{2},\frac{l-1}{2}}(\mathbb{T}^2)}
\leqslant C \sum\limits_{j = 0}^{N}\|\mu_j\|.
\end{equation}
\end{theorem}
Surely, by the norm of a measure we mean its total variation. 

There is a similar statement for the plane, which looks like this.
\begin{theorem}\label{main3}
Let $\varphi_j$, $j=1,\dots,N$, be compactly supported distributions
on the plane $\mathbb{R}^2$. Suppose that they satisfy equations
\textup{(\ref{syst})}, where $\mu_0,\dots,\mu_N$ are finite \textup{(}compactly
supported\textup{)} measures on the plane. If either $k$ or $l$ is odd, then an analog of
\textup{(\ref{est1})} holds true, namely,
\begin{equation}\label{est}
\sum\limits_{j = 1}^{N}\|\varphi_j\|_{W_2^{\frac{k-1}{2},\frac{l-1}{2}}(\mathbb{R}^2)}
\leqslant C \sum\limits_{j = 0}^{N}\|\mu_j\|.
\end{equation}
\end{theorem}

The Sobolev space on the plane in this theorem is defined as follows:
\[
W_2^{\alpha,\beta}(\mathbb{R}^2)=
\{f \in S'(\mathbb{R}^2) \colon
|\xi|^{\alpha} |\eta|^{\beta}\hat{f}(\xi,\eta) \in
L^2(\mathbb{R}^2)\}.
\]
This means that we disregard junior derivatives deliberately. Also, we note that the
two theorems remain true if both $k$ and $l$ are even, but we shall not dwell
on this in the present paper. In fact, Theorem \ref{main2} will be applied in the situation
in which $k$ and $l$ are coprime.

Taken together, Theorems \ref{main2} and \ref{main3} constitute
the second main result of the paper. They differ from classical
statements by the fact that the condition to be a measure (or an
integrable function, which is nearly the same in our setting) is imposed
on some linear combinations of derivatives of different functions rather
than on certain derivatives themselves of one function. If $k=l=1$, the
two theorems were proved in \cite{KM1} by adapting the well-known and
elegant Gagliardo--Nirenberg
method that yields a limit order embedding in Sobolev's original situation.
To be specific, we mean the inequality
\[
\|\varphi\|_{L^2(\mathbb{R}^2)}^2\leqslant
\|\partial_1\varphi\|_{L^1(\mathbb{R}^2)}\|\partial_2\varphi\|_{L^1(\mathbb{R}^2)},
\]
where $\varphi$ is a compactly supported smooth function; we remind the reader
that the key point of this method is to write two formulas representing
$\varphi$ as an ``indefinite integral" of its derivative of order 1.
However, when $k$ or $l$ is greater than $1$, that method fails totally. At present,
passage to Fourier transforms
and then forcing our way through fairly nasty improper oscillatory
integrals seems only be available in the general case. 
It should however be noted that, ideologically, the claims of
Theorems \ref{main2} and \ref{main3} can be
viewed as a blend of the embedding result in \cite{KM1} and
the well-known embedding theorems for anisotropic Sobolev
spaces, see \cite{S}, \cite{PS}.

It is natural to ask what happens if the assumption of Theorem \ref{main1}
is violated, that is, the $\Lambda$-senior parts of the differential
operators are proportional to some of them for an arbitrary choice of
the plane $\Lambda$. Here, for simplicity, we again restrict our
considerations to the case of $n=2$. 
We shall see that, under a certain ``ellipticity"
condition, the space $C^T(\mathbb{T}^2)$ is isomorphic in this situation
to $C(\mathbb{T}^2)$.\footnote{We recall that, by the Milyutin theorem, all
spaces $C(K)$ for $K$ compact metric and uncountable are mutually isomorphic.
So we can always talk about some fixed one, say, $C(\mathbb{T})$ in similar
situations} However, without this ellipticity condition nonisomorphism to
a complemented subspace of a $C(K)$-space may occur again, at least for
two different reasons. First, a change of variables may sometimes restore
the applicability of Theorem \ref{main1}.
Second, there is an effect of
somewhat different nature, related to the Cohen theorem on idempotents and
also leading to nonisomorphism.
This effect was first observed in \cite{M} in the case of dimension 3.

However, our analysis will still be not quite complete. The problems remaining are
rather of arithmetic nature. In particular, consider the collection of
two operators $T=\{\id, \partial_1+\sqrt{2}\partial_2\}$.
\textit{We do not know whether the space} $C^T(\mathbb{T}^2)$
\textit{embeds complementedly in a} $C(K)$-\textit{space}.

The paper is organized as follows. In \S\ref{S1}, we explain in detail what we know
if the assumption of Theorem \ref{main1} fails. In \S\ref{S2}, we prove Theorem \ref{main1}
on the basis of Theorem \ref{main2}. The arguments will be somewhat similar to
but more complicated than those in \cite{KM1}, \cite{KM2}. Finally,
in \S\ref{S3} we prove Theorem \ref{main2}. For that, we shall first establish
Theorem \ref{main3} (which will be space-consuming somewhat) and then reduce
Theorem \ref{main2} to it (which will not be immediate either). The proofs in \S\ref{S3}
are mainly due to the third author.

Finally, we address the reader to the monograph \cite{W} for the most part of the
Banach space theory stuff mentioned in what follows (the Milyutin theorem, $p$-absolutely
summing operators, the Grothendieck theorem, etc.).

\section{Beyond Theorem \ref{main1}}\label{S1}
\subsection{Statements}
As has already been mentioned, the ``nonisomorphism" Theorem \ref{main1}
reduces to dimension 2. In this section we aim at opposite results,
but we still restrict ourselves to dimension 2 for simplicity. As was observed
already in \cite{KM2} and \cite{M}, new subtleties may emerge in higher
dimensions.

For convenience, we restate
Theorem \ref{main1} for $n=2$. In this case the hyperplane $\Lambda$ is a straight line.
It is clear that, to verify the assumption of Theorem \ref{main1}, we
need not search through \textit{all possible} straight lines. Specifically,
consider the collection $\mathcal{M}=\{M_1,\dots,M_L\}$ of
all points on the plane whose coordinates are biindices of the differential
monomials involved in operators that determine the space in question;
then only those lines $\Lambda$
are interesting that pass through at least two points of this collection and,
surely, intersect the two positive coordinate semiaxes and have the property
that no point among the $M_s$, $s=1,\dots,L$, lies above $\Lambda$. Such a line
is said to be \textit{admissible}. Now, the desired statement looks
like this.
\begin{theorem}\label{main}
Suppose that, for at least one admissible line $\Lambda$, among the
$\Lambda$-senior parts of the operators in the collection $T=(T_1,\dots,T_l)$
there are at least two linearly independent. Then the bidual of
$C^T(\mathbb{T}^2)$ is not isomorphic to a complemented subspace of a
$C(K)$-space.
\end{theorem}

In this section we are interested in the situation when the assumption of
the above theorem is violated. Let us understand what this means. There may be
either no or finitely many admissible lines. In the latter
case, a concave broken line\footnote{We mean that it
is the graph of a concave function.} $\mathcal{Z}$ arises
from them, with nodes
$Z_t = (x_t,y_t), t = 1,\dots,K$. 
We may and do assume that the $x_t$ decrease
and the $y_t$ increase monotonically with $t$. This means that the indices
increase if we move along $\mathcal{Z}$ counterclockwise. The admissible
line $\Lambda$ passing through $Z_t$ and $Z_{t+1}$ will be marked
with the index $t$. Then the $\Lambda_t$-senior part
of each $T_j$ is the linear combination (with the same coefficients as in $T_j$)
of the differential monomials involved in
$T_j$ and such that the points corresponding to their biindices lie on the segment
$[Z_t,Z_{t+1}]$. Note that, for every $t$, the monomial
$\partial_1^{x_t}\partial_2^{y_t}$ corresponding to $Z_t$ occurs obligatory in some
senior part.

It is convenient to extend the broken line $\mathcal{Z}$ by adding to it
two points $Z_0=(x_1, 0)$ and $Z_{K+1}=(0,y_K)$ and the corresponding
segments. Then the initial broken line will be referred to as the
\textit{core} of the extended one. Each of the two additional segments
may degenerate to a point;
otherwise the first of them is part of a vertical line and the second is
part of a horizontal line; naturally, these two lines \textit{are not regarded as admissible}.

If there are no admissible lines at all, the configuration described above turns into
a broken line with three nodes $M_0,M_1,M_2$; its core is reduced to the
point $M_1$, and the other two points are the projections of $M_1$ to the coordinate
axes. Note that in this case the differential monomial corresponding to
$M_1$ must occur with a nonzero coefficient in at least one $T_j$.

We define the \textit{principal part} of $T_j$ to be the linear combination, with
the same coefficients as in $T_j$, of the differential monomials corresponding to
the points that lie in the core of $\mathcal{Z}$. We omit the easy proof of the
following statement.
\begin{lemma}\label{pp}
If the assumption of Theorem \textup{\ref{main}} is violated, then all principal parts
of the operators $T_j$ are multiples of one.
\end{lemma}

It has already been mentioned that the space $C^T(\mathbb{T}^2)$ is in
fact determined by the linear span of the collection $T$. Now, by the lemma,
if the assumption of Theorem \ref{main} fails, we can replace
the $T_j$ by their linear combinations (without changing the space)
in such a way that the only operator with nontrivial principal part
is $T_1$. If admissible lines do exist, we use the fact that
the monomial
corresponding to an arbitrary node $Z_k$ must occur in some senior part to
conclude that then the $\Lambda_t$-senior parts $S_t$ of $T_1$ are nonzero
for every $t$.  

Next, the characteristic polynomial $P_S$ of a differential operator
$S=\sum_k a_k\partial_1^{\alpha_k}\partial_2^{\beta_k}$ on the two-dimensional torus
is defined as follows:
\[
P_S(x,y)=\sum_k a_k(2\pi ix)^{\alpha_k}(2\pi iy)^{\beta_k}.
\] 
Now, if admissible lines exist, we want to impose some
\textit{ellipticity condition}. 

\begin{definition}\label{ellipticity}
We say that $S$ satisfies the ellipticity condition if
the characteristic polynomials of all senior parts $S_t$ do not
vanish on $\mathbb{R}^2$ except, possibly, for the coordinate axes.
\end{definition}

\begin{theorem}\label{suppl}
Suppose that the assumption of Theorem \textup{\ref{main}} is violated, at least one
admissible line exists and, after the modification of the collection $T$
described above, the ellipticity condition is fulfilled.
Then $C^T(\mathbb{T}^2)$ is isomorphic to $C(\mathbb{T}^2)$.
\end{theorem}
Recall that, by the Milyutin theorem, here we might replace $C(\mathbb{T}^2)$
by $C(K)$ for any uncountable compact metric space $K$.
\begin{remark}\label{suppl1}
The arguments in the proof of Theorem \ref{suppl} (to be presented later)
can be adjusted to some situations where the ellipticity condition fails.
In particular, if there are no admissible lines, the space in question
is always isomorphic to a $C(K)$-space.
\end{remark}

\subsection{Still beyond}\label{sb}The proof of Theorem \ref{suppl} will also be postponed
slightly to discuss what can happen if neither it, nor Theorem \ref{main} is
applicable. We start with an example. Consider the space determined by two operators
$T_1=\partial_1^2+2\partial_1\partial_2+\partial_2^2$ and
$T_2=a\partial_1+b\partial_2$, where $a$ and $b$ are some complex numbers.
We recall that $\partial_1$ and $\partial_2$ applied to a function $f$
on the torus are the differentiations of $f(e^{2\pi i\theta_1}, e^{2\pi i\theta_2})$
with respect to $\theta_1$ and $\theta_2$.

We see that, directly, neither of the two statements mentioned above can
tell us anything about this space. (Clearly, the first operator is equal
to $(\partial_1+\partial_2)^2$ and is not elliptic in the above sense.)
But after the change of variables
$\theta_1=t_1+t_2$, $\theta_2=t_2$, the operators $T_1$ and $T_2$ turn into
$(\partial/\partial t_2)^2$ and $(a-b)\partial/\partial t_1 + b\partial/\partial t_2$,
respectively. Theorem \ref{main} is applicable to this new collection
if $a\ne b$ (consider the
straight line passing through the points $(1,0)$ and $(0,2)$); thus, the bidual of the space
in question does not embed complementedly in a $C(K)$-space. If $a=b$,
Remark \ref{suppl1} is applicable (alternatively, the reader may consult
\cite{KSi} in this case), and we have isomorphism to $C(K)$.

We can formalize the said above as follows. Suppose there is precisely
one admissible line $\Lambda$, and suppose it is parallel to the bisector of the second
and fourth quadrants. This means that we deal with the usual rather than a
mixed homogeneity pattern. As before, suppose that only one operator in
the collection (specifically, $T_1$) has nontrivial $\Lambda$-senior part
\begin{equation}\label{sen}
L=\sum_{k=0}^\mu a_k \partial^k_1 \partial^{\mu-k}_2. 
\end{equation}
All differential monomials occurring in $L$ are of order $\mu$. We assume that all
other differential monomials involved in at least one of the $T_j$ are of order
at most $\mu-p$ for some natural number $p$, and that a monomial of
order precisely $\mu-p$ occurs in at least one
operator other than $T_1$ (say, in $T_2$). Let
$\tilde{L}$ denote the part of $T_2$ consisting of the monomials of order
precisely $\mu-p$.

We assign the polynomial $P(t)=\sum_{k=0}^{\mu} a_k t^k$ to $L$ 
(then $(2\pi ix)^{\mu}P(y/x)$ is the characteristic polynomial of $L$), and we do
similarly with $\tilde{L}$, obtaining a polynomial $\tilde{P}(t)$. Next, suppose
$P(t)$ has a real rational root $q$ of multiplicity $\alpha$, and $q$ is
a root of $\tilde{P}(t)$ of multiplicity $\beta$ (the case where $\beta =0$
is not forbidden). Finally, suppose that
$\alpha >\beta +p$ (in particular, this inequality implies that $\alpha\geqslant 2$).

Now, after the change of variables $\theta_1=t_1-qt_2$, $\theta_2 =t_2$ (this rational
substitution should be understood properly on the torus, so as to yield
an isomorphic space of smooth
functions), the operator $\partial_2-q\partial_1$ turns into $\partial/\partial t_2$.
This implies that, in the new variables, the first $\alpha$ summands
will vanish in the expression like (\ref{sen}) for $L$ and the next one will be nonzero,
and the same will happen for $\tilde{L}$ with $\alpha$ replaced by $\beta$.
Then (again in the new variables) the straight line $\Lambda$ passing through the points
$(\beta,\mu-p-\beta)$ and $(\alpha, \mu-\alpha)$ is admissible and the new operators
$T_1$ and $T_2$ have linearly independent $\Lambda$-senior parts  

The above constructions leans upon the fact that a homogeneous polynomial in two
variables expands (over $\mathbb{C}$) in a product of linear factors $x-cy$. (Moreover,
it is required that some of the numbers ``$c$" be real and rational; we shall return
to this later.) In the case of mixed
homogeneity, the situation changes. We show what may happen by considering the space
of $C^T(\mathbb{T}^2)$ with a collection $T$ consisting of only one operator $S$.
The following proposition is nearly obvious. In fact, it is valid in
arbitrary dimension and was mentioned already in \cite{KM1}.   

\begin{proposition}\label{single}
Let $E$ be the set of all roots of the characteristic polynomial $P_S$ that belong
to $\mathbb{Z}^2$. Then the space $C^{\{S\}}(\mathbb{T}^2)$ is isomorphic to the subspace
$C_E(\mathbb{T}^2)$ of $C(\mathbb{T}^2)$ that consists of all functions whose Fourier
coefficients vanish on $E$.
\end{proposition}

Now, by a commonplace in the Banach space theory (see \cite{M} for more 
explanations), if the bidual of $C_E(\mathbb{T}^2)$ embeds complemendedly in a
$C(K)$-space, then the Fourier multiplier corresponding to the set $E$ is
bounded on $C(\mathbb{T}^2)$. By the celebrated Cohen idempotent theorem
(see, e.g., the monograph \cite{GM}), this happens if and only if $E$ belongs to
the \textit{coset ring} of the group $\mathbb{Z}^2$, i.e., to the ring of sets generated
by the cosets of all subgroups of this group.

Clearly, there are operators $S$ for which $E$ does not belong to the coset
ring of $\mathbb{Z}^2$. The simplest operator of this sort (by the way, it is
mixed homogeneous) is probably
\begin{equation}\label{primer}
S_0=2\pi i\partial_1-\partial_2^2.
\end{equation}
The fact
that an infinite subset of a parabola cannot lie in the coset ring is easy;
see, however, \cite{M} for a formal pattern proving statements like that.
We emphasize that such cases of nonisomorfism should
be viewed as ``accidental" in our context
because they are not quite relevant to the differential structure. 

However, we attract the reader's attention to the interesting
example for $\mathbb{T}^3$ presented in \cite{M}. Namely, this is the
operator $\partial_1^2+\partial_2^2-\partial_3^2$, which is homogeneous
in the usual sense. The set $E$ for it consists of the Pithagorean triples
and, surely, does not belong to the coset ring of $\mathbb{Z}^3$.

It should be noted that we do not know anything about the isomorphic
type of the space generated by two operators $\id$ and $S_0$, where $S_0$
is given by (\ref{primer}). But a more challenging example is the pair
$\{ \id$, $\partial_1-\sqrt 2\partial_2 \}$. Taken alone, the second operator
of this pair generates a space isomorphic to $C(\mathbb{T}^2)$
by Proposition \ref{single}. However, it does not satisfy the ellipticity
condition discussed above (that condition required the absence of
\textit{real} and not merely \textit{integral} roots for the characteristic
polynomial). So, Theorem \ref{suppl} tells nothing about the
couple in question, nor does Theorem \ref{main}. A rational-linear change of
variables (as discussed at the beginning of this subsection) also
cannot help in this case.

\subsection{Proof of Theorem \ref{suppl}}
We address the reader to the beginning of this section to recall certain
definitions and some notation. We repeat that the differential polynomials
$T_1$,\dots,$T_l$ are composed of differential monomials that correspond
to integral points lying within the domain in the first quarter under
a concave broken line $\mathcal{Z}$. The end-points of $\mathcal{Z}$
are situated on the positive semiaxes, and the nodes of it
are denoted by $Z_j=(x_j,y_j)$, $j=0\dots,K+1$. The part
of the broken line $\mathcal{Z}$ between $Z_1$ and $Z_K$ is its core.
Only $T_1$ involves differential monomials corresponding to points
in the core of $\mathcal{Z}$. They constitute the principal part $R$ of
$T_1$:
\begin{equation}\label{type}
R = \sum\limits_{j=1}^K a_j\partial_1^{x_j}\partial_2^{y_j} +
\sum\limits_{j=0}^{K}\sum\limits_{k = 1}^{\varkappa_j - 1} b_{jk}\partial_1^{x_{j} + k\frac{x_{j+1} -
x_j}{\varkappa_j}}\partial_2^{y_{j} + k\frac{y_{j+1} - y_j}{\varkappa_j}}.
\end{equation}
Here $\varkappa_j$ is the number of subintervals into which the segment
$Z_jZ_{j+1}$ is split by points with integral coordinates (note that, for fixed $j$,
these subintervals are of equal length). The first sum
involves precisely the monomials corresponding to the vertices of
$\mathcal{Z}$ (surely, the auxiliary vertices $Z_0$ and $Z_{K+1}$ are
not taken into account if
they differ from $Z_1$ and $Z_K$). By the above discussion,
$a_j\ne 0$ for $j=1,\dots,K$.

Next, we recall that $\Lambda_j$ is the (admissible) line passing
through $Z_j$ and $Z_{j+1}$, $j=1,\dots,K$.
The $\Lambda_j$-principal part of $T_1$ is then given by
\begin{equation}
S_j = a_j\partial_1^{x_j}\partial_2^{y_j} +
a_{j+1}\partial_1^{x_{j+1}}\partial_2^{y_{j+1}} +
\sum\limits_{k = 1}^{\varkappa_j - 1} b_{jk}\partial_1^{x_{j} +
k\frac{x_{j+1} - x_j}{\varkappa_j}}\partial_2^{y_{j} +
k\frac{y_{j+1} - y_j}{\varkappa_j}}.
\end{equation}

All differential monomials that correspond to points within the
domain bounded by $\mathcal{Z}$ and the segments $OZ_0$ and $Z_{K+1}O$ (here
$O$ denotes the origin) but not lying in the core of $\mathcal{Z}$ are said to
be \textit{subordinate} to the operator $R$. We shall use the same term for
arbitrary linear combinations of such monomials. Now, Theorem \ref{suppl}
can be restated as follows.

\begin{proposition}\label{MainT}
Suppose that $R$ satisfies the ellipticity condition \textup{(}see Definition
\textup{\ref{ellipticity})} and $r_0,r_1, \dots, r_K$ are some
operators subordinate to $R$. Then the space
$C^{\{R + r_0, r_1, \dots, r_K\}}(\mathbb{T}^2)$ embeds
complementedly in $C(K)$.
\end{proposition}
\begin{remark}
In fact, standard arguments show that in our case the space is isomorphic
to $C(\mathbb{T})$. We shall not dwell on this.
\end{remark}
We shall deduce the proposition from the fact that the contribution of the $r_j$ to the norm
is negligible compared to that of $R$. In fact, this is true on a subspace of finite codimension.

For a natural number $M$, we denote by $C_M(\mathbb{T}^2)$ the subspace of $C(\mathbb{T}^2)$
consisting of all functions $f$ such that
\[
\hat{f}(m,n) = 0 \quad \mbox{for} \quad |m| \leqslant M, |n| \leqslant M.
\]
Next, we denote by $C_{M,0}$ the space of proper continuous functions satisfying the
same condition.\footnote{Recall that a function $h$ is proper if $\hat{h}(s,t)=0$ if
either $s=0$ or $t=0$.}

Clearly, $C_M(\mathbb{T}^2)$ has finite codimension in $C(\mathbb{T}^2)$ and $C_{M,0}$ has
finite codimension in the space of proper continuous functions. Similarly,
we can define the space $C_{M,0}^{\{T_1,T_2,\dots,T_n\}}(\mathbb{T}^2)$, again by imposing
the above conditions on Fourier coefficients. This subspace is easily seen to be complemented
in $C_M^{\{T_1,T_2,\dots,T_n\}}(\mathbb{T}^2)$. As a complement, we can take the sum of
a finite-dimensional space and the space $X$ of functions depending on only one
of two variables. It is easy to realize that $X$ is isomorphic to $C(\mathbb{T})$.
Thus, it suffices to show that
$C_{M,0}^{\{T_1,T_2,\dots,T_n\}}(\mathbb{T}^2)$ embeds complementedly in a $C(K)$-space.

Now, we state the required quantitative result.
\begin{proposition}\label{SecondT}
Suppose $R$ obeys the ellipticity condition and $r$ is subordinate to $R$.
Then for every $\varepsilon > 0$ there exists a natural number $M$ such that
\begin{equation}\label{subordconst}
\|rf\|_{C(\mathbb{T}^2)} \leqslant \varepsilon\|Rf\|_{C(\mathbb{T}^2)}
\end{equation}
for all $f\in C_{M,0}$.
\end{proposition}

We postpone slightly the proof of this statement to deduce
Proposition \ref{MainT} from it. It suffices to show that
$C_{M,0}^{\{R+r_0,r_1,\dots,r_K\}}(\mathbb{T}^2)$ is isomorphic to $C_{M,0}$
for large $M$, because the latter space is complemented in $C(\mathbb{T}^2)$.
But theorem \ref{SecondT} implies that, for large $M$, the norms of the spaces
$C_{M,0}^{\{R+r_0,r_1,\dots,r_K\}}(\mathbb{T}^2)$ and $C_{M,0}^{\{R\}}(\mathbb{T}^2)$
are equivalent on the set of proper trigonometric polynomials in two variables,
which is dense in each of these spaces.
So, we must establish isomorphism between $C_{M,0}^{\{R\}}(\mathbb{T}^2)$ and
$C_{M,0}$. But, plainly, the mapping $f\mapsto Rf$ from the first space to the second
is an isometry onto its image. So, it suffices to show that this image is dense
in the second space, and for this it suffices to verify that the characteristic polynomial
of $R$ does not vanish at any point outside the square $[-M,M]^2$, except for those
on coordinate axes, provided $M$ is sufficiently large. (Indeed, then all proper
trigonometric polynomials with spectrum outside this square lie in the image of
$R$.)

However, this again follows from Proposition \ref{SecondT} applied to
the operator $\id$ (the differential monomial corresponding to the origin).

Now, we start the proof of Proposition \ref{SecondT}. It is an immediate consequence of the
following  multiplier lemma.
\begin{lemma}\label{L1}
Suppose that the monomial $\partial_1^{\alpha}\partial_2^{\beta}$
is subordinate to the operator $R$. Then the Fourier
multiplier whose symbol is given by $\mu (m,n)=0$ if $m=0$, or $n=0$, or $P_R(m,n) = 0$, and otherwise
\begin{equation}
\mu (m,n) = \frac{(2\pi im)^{\alpha}(2\pi in)^{\beta}}{P_R(m,n)}
\end{equation}
is bounded from $L^1_0(\mathbb{T}^2)$ into itself and from $C_0(\mathbb{T}^2)$
into itself. Moreover, the norm of its restriction to $C_{M,0}$ \textup{(}and to the space
of functions in $L^1$ whose Fourier coefficients vanish on
$[-M,M]^2$\textup{)} tends to zero as $M \rightarrow \infty$.
\end{lemma}
\begin{remark}
It will be clear form the sequel that $P_R$ does not vanish outside $[-M,M]^2$ if $M$ is sufficiently large.
\end{remark}
We remind the reader that $P_R$ is the characteristic polynomial of $R$ (in other
words, this is its symbol if we describe the action of $R$ in terms of Fourier transforms).
Next, as usual in this paper, the subscript ``$0$" in the notation for a function class
means the subspace of proper functions in this class.

This lemma is rather standard modulo the following crucial estimate;
by the way, the estimate implies that
the denominator in the above formula does not vanish outside
the square $[-M,M]^2$ and the coordinate axes.

\begin{lemma}\label{BoundL}
If $R$ obeys the ellipticity condition, then for every node $Z_j = (x_j,y_j)$,
$j=1,\dots,N$, of the broken line $\mathcal{Z}$ we have 
\begin{equation}\label{bound}
|m|^{x_j}|n|^{y_j} \leqslant C |P_R(m,n)|
\end{equation}
for all {\bf real} $m$ and $n$ with $\max\{|m|,|n|\}$ sufficiently large, where
$C$ is independent of $m$ and $n$.
\end{lemma}
(Note that, naturally, we do not claim anything about the auxiliary nodes
$Z_0$ and $Z_{N+1}$.)

Now we show how to deduce Lemma \ref{L1} from Lemma \ref{BoundL}. We split the multiplier
into 4 pieces corresponding to the 4 coordinate quarters. It suffices to prove the boundedness
of each of these pieces separately. Since they are similar, we consider only
the multiplier $D_{\nu}$,
where $\nu (m,n)=\mu (m,n)$ if $m,n>0$ and $\nu (m,n)=0$ otherwise. Next,
clearly it suffices to estimate $D_{\nu}$ on the set of proper trigonometric polynomials
$f$ in two variables. Then the sum in the following formula is in fact finite:
\[
D_{\nu}(f)(x,y) = \sum_{m,n\geqslant 0}\nu (m,n)\hat{f}(m,n)e^{2\pi i(mx+ny)}.
\]
Summation by parts consecutively in the first and the second variable yields
\begin{equation}\label{abel}
D_{\nu}(f)(x,y) = \sum_{m,n\geqslant 0}S_{m,n}(f)\delta_x\delta_y(\nu)(m,n).
\end{equation}
Here $S_{m,n}(f)=\sum_{0\leqslant k\leqslant m}\sum_{0\leqslant j\leqslant n}
\hat f(k,j)z_1^k z_2^j$ is a partial sum of the Fourier series of $f$ and
$\delta_x$ and $\delta_y$ are the standard difference operators in the
first and the second variable. (Note that, in general, additional terms
arise under summation by parts. Here they vanish because $\nu (m,n)=0$ if
$m=0$ or $n=0$.)

Since the norm of the operator $S_{m,n}$ (not matter whether
it is considered on $C(\mathbb{T}^2)$ or on $L_1(\mathbb{T}^2)$)
is of order of $\log (m+1)\log (n+1)$ for $m,n\geqslant 0$, we see that it suffices
to prove that the series
\[
\sum_{m,n\geqslant 0}\log (m+1)\log (n+1)|\delta_x\delta_y(\nu)(m,n)|
\]
converges absolutely and its sum tends to $0$ as $M\rightarrow\infty$. (Recall that
$S_{m,n}(f)=0$ if $m$ and $n$ are smaller than $M$, so that summation in \eqref{abel}
is in fact over the indices satisfying $\max (|m|,|n|)\geqslant M$.)

The reason why this claim is true is that, when dropping the logarithmic
factors, we basically obtain an (absolutely) convergent series of inverse powers, so that
logarithms cannot make it divergent. Let us enter in the details. 

Clearly, $|\delta_x\delta_y(\nu)(m,n)|\leqslant\max_{m\leqslant s\leqslant m+1,0\leqslant t\leqslant n+1}
|\frac{\partial^2}{\partial_s\partial_t}\nu (s,t)|$. Calculating the above partial
derivative in accordance with the usual differentiation rules,
we obtain the expression $\frac{Q}{P_R^4}$, where $Q$ is some polynomial.
Let $s^pt^q$ be any nonzero monomial occurring in $Q$, and let
\begin{equation}\label{one}
\frac xa+\frac yb=1
\end{equation}
be the equation of some admissible line. It is easy to realize that then
\begin{equation}\label{four}
\frac{p+1}{a}+\frac{q+1}{b}<4.
\end{equation}
Indeed, the exponents of the monomial $(2\pi is)^{\alpha}(2\pi it)^{\beta}$
satisfy $\frac{\alpha}{a}+\frac{\beta}{b}<1$, the exponents $u,v$ of any monomial
occurring in $P_R$ satisfy $\frac{u}{a}+\frac{v}{b}\leqslant 1$, and, altogether, we
multiply four such factors ($(2\pi is)^{\alpha}(2\pi it)^{\beta}$ is involved obligatory)
but differentiate one of them in $s$ and one (which may or may not be the same)
in $t$ when forming a monomial in the numerator.

We estimate from above the quantity $|\frac{s^pt^q}{P_R(s,t)}|$ for $s\in [m,m+1]$,
$t\in [n,n+1]$. Consider the point $C=\left(\frac{p+1}{4}, \frac{q+1}{4}\right)$, then
it lies within the domain bounded by the broken line $\mathcal{Z}$ and two segments
of the coordinate axes. We claim that there exist two neighboring nodes
$Z_j=(x_j,y_j)$ and $Z_{j+1}=(x_{j+1},y_{j+1})$ of $\mathcal{Z}$, $j=1,\dots,K$ with
$\frac{p+1}{4} < x_j$ and $\frac{q+1}{4} < y_{j+1}$.

Indeed, this is clear
if $C$ lies on $\mathcal{Z}$ (but, surely, not in its core). Otherwise, consider
the lines $x=\frac{p+1}{4}$ and $y=\frac{q+1}{4}$. They hit $\mathcal{Z}$ at two points $A$ and
$B$. If the part of $\mathcal{Z}$ between $A$ and $B$ contains a link or is contained
within one link of $\mathcal{Z}$, then
the endpoints of this link fit. If this part contains precisely one node of $\mathcal{Z}$, this
node can be taken for one of the above two vertices, and either the preceding or the next node
can be taken for the other. We shall assume that the line \eqref{one} passes precisely through
these two nodes.

Now, applying Lemma \ref{BoundL}, we find
\begin{equation}\label{init}
\left|\frac{s^pt^q}{P_R(s,t)}\right|\leqslant C\frac{s^pt^q}{(s^{x_j}t^{y_j}
+ s^{x_{j+1}}t^{y_{j+1}})^4}.
\end{equation}
Next, let $\gamma >0$ be the difference
between $4$ and the left-hand side of (\ref{four}).
The quantities on the right in \eqref{init} should be multiplied by
$\log (s+1)\log (t+1)$ and then summed, and
we must show that the sum over $(s,t)\notin [-M,M]^2$ tends to $0$ as $M\rightarrow\infty$.
However, we take the liberty to act as if the logarithmic factors were absent. Let us sum the
above quantities over the pairs $(s,t)$ satisfying $s^{x_j}t^{y_j}\leqslant s^{x_{j+1}}t^{y_{j+1}}$
(the case of the opposite inequality is treated similarly); this condition means in fact that
$t\geqslant s^{\frac ab}$. So, the sum in question is dominated termwise by the double series
\[
\sum_s\left(\sum_{t\geqslant s^{a/b}}t^{q-4y_{j+1}}\right)s^{p-4x_{j+1}}.
\]
By the above discussion, $q-4y_{j+1}<-1$, whence the inner sum over $t$ is dominated by
$s^{\frac ab (q-4y_{j+1}+1)}$. Thus, we arrive at the series $\sum_s s^{\rho}$, where
\begin{multline*}
\rho=p-4x_{j+1}+\frac ab \left(q-4y_{j+1}+1\right)= \\
a\left(\frac pa -4\frac{x_{j+1}}{a}+\frac qb-4\frac{y_{j+1}}{b}+\frac 1b\right)=
a\left(\frac{p+1}{a}+\frac{q+1}{b}-4\left[\frac{x_{j+1}}{a}+\frac{y_{j+1}}{b}\right]-\frac 1a\right) \\
=-1-a\gamma.
\end{multline*}
Clearly, when we incorporate the
omitted logarithmic factors, we still can dominate the resulting series by a convergent series of
inverse powers with slightly smaller exponent.

Now, since the series converges, its sum over $(s,t)\notin [-M,M]^2$ tends to $0$, and we
are done.

It remains to prove Lemma \ref{BoundL}.

Consider two differential monomials involved in $R$ and corresponding to
the nodes $Z_i$ and $Z_j$ in the core of the broken line $\mathcal{Z}$,
that is, the monomials $\partial_1^{x_i}\partial_2^{y_i}$
and $\partial_1^{x_j}\partial_2^{y_j}$. We want to compare the magnitudes of
their characteristic polynomials. These polynomials (in fact, monomials)
are equal in the absolute value on the set where $|2\pi m|^{x_i}|2\pi n|^{y_i} = |2\pi m|^{x_j}|2\pi n|^{y_j}$,
which can be rewritten as $|2\pi m|^{c_{ij}}|2\pi n|^{-d_{ij}} = 1$, where $c_{ij}=x_i-x_j$
and $d_{ij}=y_i-y_j$. Next, their moduli are comparable on the set 
\begin{equation}
w_{ij}=\{(m,n)\in\mathbb{Z}^2\setminus [-M,M]^2 \colon \gamma_{i,j}\leqslant |m|^{c_{ij}}|n|^{-d_{ij}}\leqslant
\gamma_{i,j}^{-1}\},
\end{equation}
where $\gamma_{i,j}$ is a constant not exceeding $1$ and to be fixed later.
Naturally, $w_{ij}$ splits in four pieces (four combinations of pluses and minuses are possible
when we lift the modulus signs).
Also, $\mathbb{R}^2 \setminus [-M,M]^2$ becomes split into four
parts, each containing a ray of a coordinate axis.
In the parts containing rays of the $x$-axis, the monomial with smaller index dominates the other
(that is, for $i<j$, the algebraic monomial corresponding to $Z_i$
dominates the one liked with $Z_j$), and the opposite is true in two other
parts.

The sets $w_{ij}$ are symmetric with respect to each coordinate axis. Therefore, it suffices
to understand how they are situated relative to each other in the first quarter.
We say that a point $P \in \mathbb{R}^2$ is greater than another point $Q$
if their first coordinates are equal, they lie in one and the same quarter,
and the second coordinate of $P$ has larger absolute value than the second coordinate
of $Q$. 

\begin{proposition}
For every collection $\gamma_{i,j}$, $\gamma_{i,j} \in (0,1)$, there exists $M$ so large
that every point of $w_{ij}$ is greater than some point of $w_{kl}$, whenever $i \geqslant k, j > l$.
\end{proposition}

\begin{proof}
It suffices to verify this in the first quarter.
But on $w_{ij}$ we have $n \asymp m^{\frac{c_{ij}}{d_{ij}}}$, whereas on $w_{kl}$
we have $n \asymp m^{\frac{c_{kl}}{d_{kl}}}$. It remains to observe that
$\frac{c_{kl}}{d_{kl}} < \frac{c_{ij}}{d_{ij}}$ because the broken line $\mathcal{Z}$
is concave.
\end{proof}

Now consider the domains $w_{j,j+1}$, which ``separate" the monomials corresponding
to $Z_j$ and $Z_{j+1}$. Whatever the collection $\gamma_{j,j+1}$ is, there exists a large $M$
such that, outside $[-M,M]^2$, the sets $w_{j,j+1}$ are mutually disjoint and follow one
after another in the counterclockwise order when $j$ increases. We denote the set  
between $w_{j-1,j}$ and $w_{j,j+1}$ by $\Omega_j$.
In terms of inequalities, the definition of $\Omega_j$ looks like this:
\begin{multline}\label{Omega}
\Omega_{j} = \{(m,n)\notin [-M,M]^2 \colon \gamma_{j-1,j}\geqslant |m|^{c_{j-1,j}}|n|^{-d_{j-1,j}},\\
|m|^{c_{j,j+1}}|n|^{-d_{j,j+1}}\geqslant \gamma_{j,j+1}^{-1}\}.
\end{multline}
Now, $\Omega_j$ is the set where the monomial corresponding to $Z_j$ dominates
the other monomials. Here is the precise statement.

\begin{proposition}
Let $\partial_1^x\partial_2^y$ be a monomial involved in $R$ but not corresponding
to $Z_j$ \textup{(}it may correspond either to some other node or to a
point on some segment of the broken line\textup{)}, and let $\lambda$ be a positive number.
If $\gamma_{j-1,j}$ and $\gamma_{j,j+1}$ are sufficiently small and $M$ is sufficiently
large, we have
\begin{equation}\label{omegaBound}
|m|^{x_j}|n|^{y_j} \geqslant \lambda|m|^x|n|^y\quad\hbox{for}\quad (m,n) \in \Omega_j.
\end{equation}
\end{proposition}

\begin{proof}
We rewrite the inequality to be verified:
\[
(x_j - x) \log|m| + (y_j - y)\log|n| \geqslant \log\lambda.
\]
The point $(\log|m|,\log|n|)$ lies in the domain
\begin{multline*}
\{(x,y) \notin [-\log M,\log M]^2 \colon\\ \log\gamma_{j-1,j} \geqslant c_{j-1,j}u - d_{j-1,j}v;\quad
c_{j,j+1}u - d_{j,j+1}v \geqslant -\log\gamma_{j,j+1}\}.
\end{multline*}
Consider the case where $(x,y)$ has smaller argument than $Z_j$,
the symmetric case is treated similarly. Since the broken line $\mathcal{Z}$
is concave, we have $\frac{x_j - x}{y_j - y} + \frac{c_{j-1,j}}{d_{j-1,j}} \geqslant 0$.
The definition of $\Omega_j$ shows that
\begin{equation*}
\log|n| \geqslant \frac{-\log\gamma_{j-1,j} + c_{j-1,j}\log|m| }{d_{j-1,j}},
\end{equation*}
that is,
\begin{align}
\label{nonnegative}(x_j - x) \log|m|& + (y_j - y)\log|n|\geqslant\\
(x_j - x) \log|m|& + (y_j - y)\frac{-\log\gamma_{j-1,j} + c_{j-1,j}\log|m| }{d_{j-1,j}}\notag\\
\geqslant (y_j-y)\frac{-\log\gamma_{j-1,j}}{d_{j-1,j}},\notag
\end{align}
because the coefficient of $\log|m|$ is nonnegative. Taking $\gamma_{j-1,j}$
sufficiently small, we can make this quantity greater than $\log\lambda$.
\end{proof}

This proposition allows us to prove inequality \eqref{bound} on $\Omega_j$
if $\gamma_{j,j-1}$ is small.
Indeed, recalling formula \eqref{type} for $R$, we see that we must prove the inequality
\begin{multline*}
|2\pi m|^{x_i}|2\pi n|^{y_i} \leqslant\\ C \left|\sum\limits_{\theta =1}^K a_{\theta}(2\pi im)^{x_{\theta}}
(2\pi in)^{y_{\theta}} +
\sum\limits_{\theta =0}^{K}\sum\limits_{k = 1}^{\varkappa_{\theta} - 1}
b_{\theta k}(2\pi im)^{x_{\theta} + k\frac{x_{\theta +1} - x_{\theta}}{\varkappa_{\theta}}}(2\pi in)^{y_{\theta}
+ k\frac{y_{\theta +1} - y_{\theta}}{\varkappa_{\theta}}}\right|.
\end{multline*}
By \eqref{omegaBound}, on $\Omega_j$ we can replace the monomial $|m|^{x_i}|n|^{y_i}$ on the left
by $|m|^{x_j}|n|^{y_j}$. Now, we choose $\lambda$ greater than
$a_j^{-1}2K\max_{\theta,k}\varkappa_{\theta}\max(|a_{\theta}|,|b_{\theta k}|)$,
after which we choose the numbers ``$\gamma$"
so as that to ensure \eqref{omegaBound} with this $\lambda$ for all monomials involved in $R$.
Then
\begin{multline*}
\left|\sum\limits_{\theta=1}^K a_{\theta}(2\pi im)^{x_{\theta}}(2\pi in)^{y_{\theta}} +
\sum\limits_{\theta=0}^{K}\sum\limits_{k = 1}^{\varkappa_{\theta} - 1} b_{\theta k}(2\pi im)^{x_{\theta}
+ k\frac{x_{\theta+1} -
x_{\theta}}{\varkappa_{\theta}}}(2\pi in)^{y_{\theta} + k\frac{y_{\theta+1} -
y_{\theta}}{\varkappa_{\theta}}}\right|\geqslant\\
|a_j(2 \pi i m)^{x_j}(2\pi in)^{y_j}| -\\ \left|\sum\limits_{\theta=1, \theta
\ne j}^K a_{\theta}(2\pi im)^{x_{\theta}}(2 \pi i n)^{y_{\theta}} +
\sum\limits_{\theta=0}^{K}\sum\limits_{k = 1}^{\varkappa_{\theta} - 1} b_{\theta k}(2\pi im)^{x_{\theta} +
k\frac{x_{\theta+1} - x_{\theta}}{\varkappa_{\theta}}}(2\pi in)^{y_{\theta} + k\frac{y_{\theta+1} -
y_{\theta}}{\varkappa_{\theta}}}\right|\\
\geqslant\frac{|a_j|}{2}|m|^{x_j}|n|^{y_j}.
\end{multline*}

It remains to ensure \eqref{bound} on the domains $w_{j,j+1}$. The arguments will be similar.

\begin{proposition}
Let $\partial_1^x\partial_2^y$ be a monomial of $R$ with $(x,y)$ not belonging
to the segment $Z_jZ_{j+1}$ \textup{(}$(x,y)$ may be a node in the core of $\mathcal{Z}$
or may belong to some link of this core\textup{)}, and let $\lambda >0$.
If $M$ is sufficiently large, for the $\varkappa_j$ fixed above we have
\begin{equation}\label{wBound}
|m|^{x_{j} + k\frac{x_{j+1} - x_j}{\varkappa_j}}|n|^{y_{j} +
k\frac{y_{j+1} - y_j}{\varkappa_j}} \geqslant \lambda|m|^x|n|^y \quad \hbox{for}\quad (m,n) \in w_{j,j+1}.
\end{equation}
\end{proposition}

Note that, on $w_{j,j+1}$, all algebraic monomials corresponding to points on the line $Z_jZ_{j+1}$
are comparable.

The proposition is proved much as the preceding one. The only difference is that
the coefficient of $\log |m|$ in \eqref{nonnegative} is strictly positive, because
this time the point $(x,y)$ lies strictly below the line $Z_jZ_{j+1}$.
So, we need not impose additional assumptions on $\gamma_{j, j+1}$, it suffices to merely
take $|m|$ sufficiently large
(this is important, because the constant $\gamma_{j, j+1}$ has already been fixed).

Now, we proceed to the verification of inequality \eqref{bound} on $w_{j,j+1}$.
First, the monomial $|m|^{x_i}|n|^{y_i}$ under study can be replaced by
any monomial of the form
$|m|^{x_{j} + k\frac{x_{j+1} - x_j}{\varkappa_j}}|n|^{y_{j} + k\frac{y_{j+1} - y_j}{\varkappa_j}}$
by the last proposition and the above observation that such monomials are comparable.
Second, we have the estimate

\begin{multline*}
\big|a_j (2\pi i m)^{x_j}(2\pi i n)^{y_j} + a_{j+1}(2\pi i m)^{x_{j+1}}(2\pi i n)^{y_{j+1}} +\\
\sum\limits_{k = 1}^{\varkappa_j - 1} b_{jk}(2\pi i m)^{x_{j} + k\frac{x_{j+1}
- x_j}{\varkappa_j}}(2\pi i n)^{y_{j} + k\frac{y_{j+1} - y_j}{\varkappa_j}}\big|=\\
\left|(2\pi i m)^{x_{j}}(2\pi i n)^{y_{j}}\prod\limits_{\theta=0}^{\varkappa_j+1}(1 +
\xi_{\theta}(2\pi i m)^{\frac{-x_{j+1} + x_{j}}{\varkappa_{j}}}(2\pi i n)^{\frac{y_{j+1}
- y_{j}}{\varkappa_{j}}})\right| \geqslant \\
C\left|\prod\limits_{{\theta}=0}^{\varkappa_j+1}
\Im(i^{\frac{y_{j+1} - y_{j} -x_{j+1} +
x_{j}}{\varkappa_{j}}}\xi_{\theta})\right| |m|^{x_j}|n|^{y_j}.
\end{multline*} 
The first identity is merely a factorization of the polynomial, and the constant in the subsequent inequality
is nonzero by the ellipticity condition (the numbers
$i^{\frac{y_{\theta+1} - y_{\theta} -x_{\theta+1} + x_{\theta}}{\varkappa_{\theta}}}\xi_{\theta}$ are not pure real)
and the fact that the quantity
$|m|^{\frac{-x_{j+1} + x_j}{\varkappa_j}}|n|^{\frac{y_{j+1} - y_j}{\varkappa_j}}$
is bounded away from zero on $w_{j,j+1}$. It follows that for the $\Lambda_j$-senior
part $S_j$ of $R$ (recall that $\Lambda_j$ includes the link $Z_jZ_{j+1}$) and for every $\lambda$
we can choose $M$ so large that, on $w_{j,j+1}$, we shall have
\begin{equation*}
|P_R(m,n)| \geqslant \frac{1}{2}\Big|S_j (m,n)\Big| \geqslant C|m|^{x_j}|n|^{y_j}.
\end{equation*}
This completes the proof.
\subsection{Other isomorphism cases} The idea of domination allows us to establish isomorphism
even in some cases when the ellipticity condition is violated. The simplest one is the case
mentioned in Remark \ref{suppl1}.
\begin{proposition}
If there are no admissible lines, the space $C^T(\mathbb{T}^2)$ is complemented in a $C(K)$-space.
\end{proposition}
In fact, the space in question is isomorphic to $C(\mathbb{T})$, but we refrain from discussing this.
\begin{proof} As previously, it suffices to prove that the subspace $C^T_0(\mathbb{T}^2)$ of
admissible functions is complemented in a $C(K)$-space. We may and do assume that $T_1=R+r$
where $R$ is a differential monomial and both $r$ and all operators $T_2,\dots,T_j$ are
subordinate to $R$. It suffices to show that, for some $M$, the norm of $C^T_0(\mathbb{T}^2)$
is equivalent to that of $C^{\{R\}}_0(\mathbb{T}^2)$ on the set of proper functions that have no
spectrum in the square $[-M,M]\times [-M,M]$. But it is easy to see that if a differential
monomial $\rho$ is subordinate to $R$, then $\|\rho f\|_{C(K)}\leqslant\varepsilon
\|Rf\|_{C(K)}$ for every such $f$ if $M$ is sufficiently large.
\end{proof}

We present yet another related example of isomorphism.
Let $e_1, e_2, \ldots, e_n$ be pairwise nonproportional rational vectors on the plane.
In what follows, we need a different notion of a proper function.
Specifically, a function is said to be \textit{quite proper} if its
Fourier coefficients vanish at all points of $\mathbb{Z}^2$ that lie
on the straight lines generated by the above vectors.
In \cite{KM1} it was explained that passage to quite proper functions
does not change the isomorphic type of the spaces we treat here. 
For any quite proper function $f \in C^T(\mathbb{T}^2)$, we can find a function
$g \in C^T(\mathbb{T}^2)$ such that $\partial_{e_k}g=f$ ($\partial_{e_k}$ denotes
a directional derivative) and $\|g\|_{\infty}\leqslant C\|f\|_{\infty}$. (By
$\|\cdot\|_{\infty}$ we mean the supremum norm.)
This operation will be called formal integration in the direction
$e_k$, but in fact it is done on the level of Fourier coefficients.
Staying within quite proper functions, we avoid division of nonzero Fourier
coefficients by zero under this operation.

Now, suppose that the family $T$ contains the operator
$T_1=\partial_{e_1} \partial_{e_2} \ldots \partial_{e_n}$ and all other
operators in the family are of order strictly smaller than $n$.
\textit{Then} $C^{(T)}(\mathbb{T}^2)$ \textit{is isomorphic to}
$C(\mathbb{T}^2)$. \textit{On the quite proper functions, an isomorphism
is given by the mapping} $f \mapsto T_1f$.

For the proof, it suffices to show that $\|T_jf\|_{\infty}\leqslant C\|T_1f\|_{\infty}$
for all $T_j\in T$ if $f$ is quite proper, with $C$ independent of $f$.
To do this, we need a simple algebraic lemma.

\begin{lemma}Let $P(x)=(x-x_1)(x-x_2) \ldots (x-x_n)$ be a polynomial of degree
$n$ that has $n$ pairwise different roots. Then its monic divisors of degree
$n-1$ form a basis in the linear space of polynomials of degree at most $n-1$.
\end{lemma}

\begin{proof}Let $P_k(x)$ denote the polynomial $\frac{P(x)}{(x-x_k)}$. When $k$ runs
from $1$ to $n$, we obtain all divisors of $P(x)$ of degree $n-1$.
It suffices to prove that they are linearly independent.
Suppose the contrary, let $\lambda_1 P_1(x) + \lambda_2 P_2(x) +\ldots +\lambda_n P_n(x)=0$
for some coefficients $\lambda_1,\lambda_2,\ldots,\lambda_n$.
Putting $x=x_k$, we obtain $\lambda_k=0$ because all roots are different.
\end{proof}

Now, we prove the above claim. We denote by $T_j^{(n-1)}$ the sum of all terms
of order $n-1$ that occur in $T_j \in T$  (some of these ``senior parts" of the $T_j$ may be
equal to zero). Next, denote by $T_{1k}$ the operator
$\frac{T_1}{\partial_{e_k}}$ (i.e., the product of all directional derivations
$\partial_{e_i}$ for $i \ne k$). We observe that, on the quite proper functions
$f$, we have $\|T_{1k}f\|_{\infty} \leqslant C\|T_1f\|_{\infty}$ by directional formal integration.
But by the lemma, each operator $T^{(n-1)}_{j}$ is a linear combination of the $T_{1k}$
whence we obtain a similar estimate for all $T^{(n-1)}_{j}$.

Now we extract the parts of order $n-2$ from all $T_j$, and estimate them in a similar
way by the same method. Then we do the same in the order $n-3$, etc.

\section{Nonisomorphism}\label{S2}

In this section, we deduce Theorem \ref{main} from Theorem \ref{main2}. This will be an
improvement of similar arguments in \cite{KM1} and \cite{KM2}. We begin with
a solvability condition for system (\ref{syst}). We observe that, since
all $\varphi_j$ in this system are assumed to be proper, all measures $\mu_j$
must also be proper.

\begin{lemma}\label{ann}
Suppose $\mu_0,\dots,\mu_N$ are proper distributions on the torus and
$k,l\in\mathbb{N}$. Then system \eqref{syst} admits a solution in proper
distributions $\varphi_0,\dots,\varphi_N$ if and only if
\begin{equation}\label{annu}
\sum_{j=0}^N\partial_1^{jk}\partial_2^{(N-j)l}\mu_j=0.
\end{equation}
Moreover, if \eqref{annu} is fulfilled, this solution is unique.
\end{lemma}
\begin{proof}The implication (\ref{syst})$\Rightarrow$(\ref{annu}) is easy.
The converse implication is proved by induction on $N$. The base ($N=1$) looks like
this: if two proper distributions $u$ and $v$ on $\mathbb{T}^2$
satisfy
\begin{equation}\label{baza}
\partial_2^l u+\partial_1^k v=0,
\end{equation}
then there is a proper distribution $w$
with $u=-\partial_1^k w$ and $v=\partial_2^l w$. This can be verified, e.g., by
considering Fourier coefficients. To pass from $N-1$ to $N$, we rewrite
(\ref{annu}) in the form (\ref{baza}) with $u=\partial_2^{(N-1)l}\mu_0$ (then it
is clear what $v$ is), find $w$ as above, and invoke the inductive
hypothesis.

Uniqueness is transparent from the above arguments. Alternatively (which
is in fact the same), it is easy to see directly from system \eqref{syst}
that the Fourier coefficients of the $\varphi_j$ corresponding to all couples
of nonzero integers are determined by those of the $\mu_j$.
\end{proof}
\subsection{Several reductions}Now, suppose we are under the assumptions
of Theorem~\ref{main}. The space $C_0^T(\mathbb{T}^2)$ is complemented in
$C^T(\mathbb{T}^2)$, so it suffices to prove that the bidual of
$C_0^T(\mathbb{T}^2)$ does not embed complementedly in a $C(K)$-space. Next,
let the admissible line $\Lambda$ mentioned in
Theorem \ref{main} be given by the equation $\frac xa +\frac yb =1$.
We claim that there is no loss of generality in assuming that $a$ and $b$
are natural numbers. Indeed, $\Lambda$ must contain two points $(r_1,r_2)$
and $(\rho_1,\rho_2)$ with nonnegative integral coordinates, and we may assume
that $r_1>\rho_1$ and $r_2<\rho_2$. Then the equation of $\Lambda$ can also be
written in the form $\frac{x-\rho_1}{r_1-\rho_1}=\frac{y-\rho_2}{r_2-\rho_2}$, or
\[
\frac{x}{r_1-\rho_1}+\frac{y}{\rho_2-r_2}=\frac{\rho_1}{r_1-\rho_1}+\frac{\rho_2}{\rho_2-r_2}.
\]
Next, if $u,v\in\mathbb{N}$, then the spaces
\[
C_0^{\{T_1,\dots,T_l\}}(\mathbb{T}^2)\quad\mbox{and}\quad
C_0^{\{T_1\partial_1^u\partial_2^v,\dots,T_l\partial_1^u\partial_2^v\}}(\mathbb{T}^2)
\]
are isomorphic. Specifically, the operator $f\mapsto\partial_1^u\partial_2^v f$
is an isomorphism from the second space onto the first. Passage to the new space
leads to shifting the line $\Lambda$ by the vector $(u,v)$. The equation of the shifted
line is 
\[
\frac{x}{r_1-\rho_1}+\frac{y}{\rho_2-r_2}=\frac{\rho_1+u}{r_1-\rho_1}+\frac{\rho_2+v}{\rho_2-r_2}.
\]
The lengths of the segments cut by the new line from the $x$- and $y$-axes are equal to
\[
\rho_1+u+(\rho_2+v)\frac{r_1-\rho_1}{\rho_2-r_2}\,\,\,\mbox{ and }\,\,\,
(\rho_1+u)\frac{\rho_2-r_2}{r_1-\rho_1}+\rho_2+v.
\]
Clearly, these numbers are integers for some choice of $u$ and $v$.

So, we assume that $a,b\in\mathbb{N}$. Let $N$ be the greatest common divisor of
$a$ and $b$. We put $m=a/N$, $n=b/N$ (thus, $m$ and $n$ are coprime). All points
on $\Lambda$ whose coordinates are nonnegative integers are of the form
$(jm,(N-j)n),\,\,\, j=0,\dots, N$.

\subsection{Annihilator and the Grothendieck theorem}Consider the natural
embedding $f\mapsto\{T_1f,\dots,T_lf\}$ of the space
$C_0^{\{T_1,\dots,T_l\}}(\mathbb{T}^2)$ into the direct sum of $l$ copies
of the space $C(\mathbb{T}^2)$, and let $\mathcal{X}$ be the image of
$C_0^{\{T_1,\dots,T_l\}}(\mathbb{T}^2)$. By a Banach space theory commonplace,
if the bidual of $C_0^{\{T_1,\dots,T_l\}}(\mathbb{T}^2)$ is isomorphic to
a complemented subspace of a $C(K)$-space, then the annihilator of
$\mathcal{X}$ in $C(\mathbb{T}^2)\oplus\dots\oplus C(\mathbb{T}^2)$ is
complemented in the dual space, which consists of $l$-tuples
of measures on $\mathbb{T}^2$\footnote{We recall the spell proving this. Let $E$ be a Banach space and $F$ its closed subspace. If $F^{\ast\ast}$ embeds complementedly in a $C(K)$-space, then $F^{\ast}$ is an $\mathcal{L}^1$-space. Trivially, the canonical surjection $\pi$ from $E^{\ast}\to F^{\ast}$ admits a section on each $l^1_n$-subspace of $F^{\ast}$. Since $\pi$ is $w^{\ast}$-continuous, taking a limit point we obtain a global section for $\pi$. So, the kernel of $\pi$ (equal to the annihilator of $F$) is complemented in $E^{\ast}$. Consult \cite{LR} for the omitted definitions and details.}. By the Grothendieck theorem, then an arbitrary
bounded linear operator from $\mathcal{X}^{\bot}$ to a Hilbert space
would have been $1$-absolutely summing, and, \textit{a fortiori},
$2$-absolutely summing. We shall show this is not the case. (The definition of a
$2$-absolutely summing operator will be given later, when it is really used.)

In order to construct an operator with the required property, we need to understand
the structure of the annihilator $\mathcal{X}^{\bot}$. First, we observe that we can
do all constructions within proper measures. Indeed, to any measure $\rho$
on $\mathbb{T}^2$, we can assign its \textit{proper part} whose Fourier
coefficients are equal to $\hat{\rho}(s,t)$ if $s\ne 0$ and $t\ne 0$, and are zero
otherwise. It is easy to realize that this operation is a bounded projection
on the space of measures. Moreover, \textit{if a collection of measures
annihilates $\mathcal{X}$, then so does the collection of their proper parts}.
In what follows, we shall work only with such collections of proper parts.
The operations described below will not produce ``improper" objects from
such collections. 

Next, let $\tau_j$ be the $\Lambda$-senior
part of the operator $T_j$, and let $\sigma_j=T_j-\tau_j$. A collection
$(\nu_1,\dots,\nu_l)$ of (proper) measures on $\mathbb{T}^2$ belongs to $\mathcal{X}^{\bot}$ if
\begin{equation}\label{annul}
0=T_1^{\ast}\nu_1+\dots+T_l^{\ast}\nu_l=\tau_1^{\ast}\nu_1+\dots+\tau_l^{\ast}\nu_l
+\sigma_1^{\ast}\nu_1+\dots+\sigma_l^{\ast}\nu_l.
\end{equation}
We want to modify the collection $T$ of differential operators without changing
the space, by using the fact that the space depends only on the linear span of the $T_j$. By
assumption, there are at least two linearly independent operators among
$\tau_1^{\ast},\dots,\tau_l^{\ast}$. We write
\begin{equation}\label{matrix}
\tau_s^{\ast}=\sum_{j=0}^N a_{sj}\partial_1^{jm}\partial_2^{(N-j)n}
\end{equation}
and make two steps in reshaping the matrix $\{a_{sj}\}$ to a diagonal form.
Let $j_0$ be the smallest index for which $a_{sj_0}\ne 0$ for at least one $s$.
By reindexing, we may assume that $s=1$; replacing $T_1$ by its multiple, we may
assume that $a_{1j_0}=1$. Next, subtracting a multiple of $T_1^{\ast}$ from each
$T_s^{\ast}$ with $s>1$, we can ensure that $s_{sj_0}=0$ for $s>1$. By assumption,
after that an index $j_1$ with $j_1>j_0$ must exist such that $a_{sj_1}\ne 0$ for
some $s>1$. There is no loss of generality in assuming that $j_1$ is the smallest
index with this property, $s=2$, and $a_{2j_1}=1$.

Returning to formula (\ref{annul}), we assume that all these modifications have
already been done. Now, we want to eliminate the junior terms $\sigma_s^{\ast}\nu_s$
in this formula. Let $\alpha,\beta$ be two nonnegative integers satisfying
$\frac{\alpha}{a}+\frac{\beta}{b}<1$. There is a standard way to express the
restriction of the differential monomial $\partial_1^{\alpha}\partial_2^{\beta}$
to the space of proper functions in terms of $\partial_1^a$ and $\partial_2^b$.
Specifically, for $u\ne 0$ and $v\ne 0$ we have
\begin{equation}\label{multip}
\widehat{\partial_1^{\alpha}\partial_2^{\beta}f}(u,v)=
\frac{(iu)^{\alpha +a}(iv)^{\beta}}{(iu)^{2a}\pm (iv)^{2b}}\widehat{\partial_1^a f}(u,v)\pm
\frac{(iu)^{\alpha}(iv)^{\beta +b}}{(iu)^{2a}\pm (iv)^{2b}}\widehat{\partial_2^b f}(u,v).
\end{equation}
The sign $+$ or $-$ must be one and the same at all occasions; it is determined by the
condition $(-1)^a=\pm (-1)^b$, then the denominators do not vanish anywhere except
the point $(0,0)$.
\begin{lemma}\label{L2}
The Fourier multiplies $I_{\alpha\beta}$ and $J_{\alpha\beta}$ with symbols
\begin{equation}\label{multip1}
\frac{(iu)^{\alpha +a}(iv)^{\beta}}{(iu)^{2a}\pm (iv)^{2b}}
\mbox{ and }\frac{(iu)^{\alpha}(iv)^{\beta +b}}{(iu)^{2a}\pm (iv)^{2b}}
\end{equation}
are bounded on $L^1_0 (\mathbb{T}^2)$, consequently, they take proper measures
to \textup{(}proper\textup{)} measures.
\end{lemma}

This lemma is a partial case of Lemma \ref{L1}. It suffices to take $(iu)^{2a}\pm (iv)^{2b}$ for the role of $P_R$.

It should be noted that, in fact, the multipliers in question take also $L^1 (\mathbb{T}^2)$
to itself (consequently, send measures to measures). This is also easy.

It is convenient to write formula (\ref{multip}) in the form
\begin{equation}\label{multip2}
\partial_1^{\alpha}\partial_2^{\beta}=\partial_1^aI_{\alpha\beta}
+\partial_2^bJ_{\alpha\beta}.
\end{equation}

We apply (\ref{multip2}) to all differential monomials occurring in some
$\sigma_j^{\ast}$ in (\ref{annul}). Then the sum $\sum_j\sigma_j^{\ast}\nu_j$ becomes
transformed to $\partial_1^a\lambda_1+\partial_2^b\lambda_2$, where $\lambda_1$
and $\lambda_2$ are measures obtained from $\nu_1,\dots,\nu_l$ by certain linear
operations. After that, we regroup the summands in (\ref{annul}) assembling together
everything that involves each particular differentiable monomial
$\partial_1^{jm}\partial_2^{(N-j)n}$ ($j=1,\dots,N$). So, in the resulting
expression this differential monomial will be applied to a certain measure $\mu_j$,
$j=0,\dots,N$, and the collection of these measures satisfies equation
(\ref{annu}). Clearly, the $\mu_j$ depend linearly on $\nu_1,\dots,\nu_l$ and
are all proper.

Now, Lemma \ref{ann} shows that  system (\ref{syst}) admits a unique solution
$\varphi_0,\dots,\varphi_N$ in proper distributions, in fact, these distributions
belong to the Sobolev space $W_2^{\frac{k-1}{2},\frac{l-1}{2}}(\mathbb{T}^2)$
(see (\ref{sobolev})) by Theorem \ref{main2}.

So, a bounded linear operator from $\mathcal{X}^{\bot}$ to the Hilbert space
$W_2^{\frac{k-1}{2},\frac{l-1}{2}}(\mathbb{T}^2)$ has arisen. We summarise how
it acts. Given a collection of measures $\mathcal{X}^{\bot}$, we assign to it the
collection $\{\nu_1,\dots,\nu_l\}\in\mathcal{X}^{\bot}$ of their proper parts; this
collection gives rise to measures $\{\mu_0,\dots,\mu_N\}$ satisfying \eqref{annu},
from which we obtain functions $\varphi_j$ belonging to the Sobolev space mentioned
above. As was mentioned at the beginning of this subsection, should the bidual
of $C_0^T(\mathbb{T}^2)$ be complemented in some $C(K)$-space, this operator would
have been $2$-absolutely summing.

\subsection{Contradiction}We explain how to show that the operator constructed above
is not $2$-absolutely summing. By definition, an operator $S\colon E\to G$ is
$2$-absolutely summing if it takes weakly $2$-summable sequences to $2$-summable
sequences. A sequence is said to be $2$-summable if the squares of the norms of
its elements form a convergent series. A sequence $\{x_j\}_{j\in E}$ is said to be
weakly $2$-summable if the series $\sum_j |F(x_j)|^2$ converges for an arbitrary
bounded linear functional $F$ on $E$. Clearly, a bounded orthogonal sequence in
a Hilbert space is weakly $2$-summable. This allows us to exhibit weakly $2$-summable
sequences in spaces of measures in the following way. Suppose measures $\sigma_k\in
M(K)$ are all absolutly continuous with respect to one measure $\sigma\in M(K)$ and
their densities form a bounded orthogonal system in $L^2(\sigma)$; then the $\sigma_k$
form a weakly $2$-summable sequence in $M(K)$. Indeed, the mapping $f\mapsto fd\sigma$
is continuous from $L^2(\sigma)$ to $M(K)$, and the property of being a weakly $2$-summable
sequence survives under the action of a bounded linear operator.

To arrive at the required contradiction, we shall construct a sequence of elements
of $\mathcal{X}^{\bot}$ enumerated by couples $(p,q)$ of natural numbers. It will be
of the form
\begin{equation}\label{contr}
\{\nu_1^{(p,q)},\dots,\nu_l^{(p,q)}\}=
\{z_1^pz_2^qd\lambda,c_{pq}z_1^pz_2^qd\lambda,0,\dots,0\}.
\end{equation}
Here $\lambda$ is the normalized Lebesgue measure on the two-dimensional torus.
The coefficients $c_{pq}$ will be uniformly bounded, so the above discussion shows
that \eqref{contr} is a weakly $2$-summable sequence in the space of $l$-tuples of
measures and, consequently, in its subspace $\mathcal{X}^{\bot}$. For technical
reasons (see below) the indices will be subject to the condition
\begin{equation}\label{uslovie}
\frac{\delta}{2}q^l\leqslant p^k\leqslant\delta q^l \mbox{ and } p\geqslant C,
\end{equation}
where $\delta$ is sufficiently small and $C$ is sufficiently large.

We shall not need know much about the images $\{\varphi^{(p,q)}_j\}_{j=0,\dots,N}$
of these $l$-tuples of measures
under the operator described above. Specifically, let $j_0$ be the index that
arose under modification of the collection $T$ (see formula \eqref{matrix} and the
explanations after it). It will be ensured that
\begin{equation}\label{snizu}
\varphi^{(p,q)}_{j_0}=(ip)^{-k}\gamma_{p,q}z_1^pz_2^q,\mbox{ where } \inf_{p,q}
|\gamma_{(p,q)}|>0.
\end{equation}
Then the series
\[
\sum_{p,q}\|\varphi_{j_0}^{p,q}\|^2_{W_2^{\frac{k-1}{2},\frac{l-1}{2}}(\mathbb{T}^2)}
\]
(summation is over $(p,q)$ satisfying \eqref{uslovie}) diverges. Indeed, clearly, by
\eqref{snizu}, the norms under
the summation sign dominate the quantities $p^{k-1}q^{l-1}p^{-2k}$. By condition
\eqref{uslovie}, for every fixed $p\geqslant C$ each admissible value of $q$
is roughly $p^{\frac kl}$, and the number of admissible $q$ is also
roughly $p^{\frac kl}$. Thus, summation of $q^{l-1}$ over admissible $q$ for $p$ fixed
yields roughly $p^{\frac kl}p^{(l-1)\frac kl}=p^k$, so that the sum of the above quantities
over $p,q$ satisfying \eqref{uslovie} dominates the sum $\sum_{p\geqslant C}p^{-1}=\infty.$

It remains to exhibit elements of the form \eqref{contr} in the annihilator that satisfy
the conditions listed above. For this, we must trace what the corresponding
measures $\mu_0,\dots,\mu_N$ (see \eqref{annu}) may look like. First, the measures
$\mu_0$ and $\mu_N$ may
involve summands that have arisen from the junior parts of the operators $T_j^{\ast}$ in
accordance with formula \eqref{multip2}. These summands are of the form
$\xi_{pq}z_1^pz_2^q+\eta_{pq}c_{pq}z_1^pz_2^q$ for $\mu_0$ and
$\rho_{pq}z_1^pz_2^q+\varkappa_{pq}c_{pq}z_1^pz_2^q$ (with other coefficients) for
$\mu_N$. The numbers $\xi_{pq}$, $\eta_{pq}$, $\rho_{pq}$, and $\varkappa_{pq}$ can be
expressed in terms of the quantities \eqref{multip1} with $u$ and $v$ replaced by
$p$ and $q$.
\begin{lemma}\label{decay}
For some $\varepsilon >0$, we have
\[
\xi_{pq},\, \eta_{pq},\, \rho_{pq},\, \varkappa_{pq}=O(p^{-\varepsilon}).
\]
\end{lemma}
We remind the reader that we have imposed the condition $p^k\asymp q^l$.
\begin{proof}
Consider, for example, the first of the quantities \eqref{multip1} with $p$
in place of $u$ and $q$ in place of $v$. Since $\frac kl =\frac ab$,
we see that, under our assumptions, this quantity has the same order of
magnitude as
\[
\frac{p^{\alpha+a}p^{\frac kl\beta}}{p^{2a}+p^{2b\frac kl}}=
\frac{p^{\alpha+a}p^{\frac ab\beta}}{2p^{2a}}=\frac 12 p^{(-a+\alpha+\frac ab \beta)}
=\frac 12 p^{-a(1-\frac{\alpha}{a}-\frac{\beta}{b})}.
\]
It remains to recall that $\frac{\alpha}{a}+\frac{\beta}{b}<1$ and that the set of the couples
$(\alpha,\beta)$ is finite.
\end{proof}

Now, we want to pay attention to the (proper) solution of system \eqref{syst} if the
collection $\{\mu_0,\dots,\mu_n\}$ has arisen from a collection like in \eqref{contr}.
The calculations depend on the specific position of the indices $j_0 <j_1$ (that arose when we
modified the operators $T_j$) in the
interval $[0,N]$. We present the details for two cases; in the other cases, combination of
the arguments below is required.

\textbf{Case 1:} $j_0=0, j_1=N$. Then \eqref{syst} acquires the form
\begin{gather}
-\partial_1^k\varphi_1=z_1^pz_2^q(1+\xi_{pq}+\eta_{pq}c_{pq});\label{81}\\
\partial_2^l\varphi_j-\partial_1^k\varphi_{j+1}=a_{1j}z_1^pz_2^q,\quad j=1,\dots,N-1;\label{82}\\
\partial_2^l\varphi_N=z_1^pz_2^q(a_{1N}+c_{pq}+\rho_{pq}+\varkappa_{pq}c_{pq}).\label{83}
\end{gather}
Resolving equation \eqref{81}, we find
\begin{equation}\label{9}
\varphi_1=-(ip)^{-k}(1+\xi_{pq}+\eta_{pq}c_{pq})z_1^pz_2^q.
\end{equation}
If $N=1$ (that is, $a$ and $b$ are coprime), then all equations \eqref{82} are absent,
and \eqref{83} yields immediately an equation for $c_{pq}$:
\[
-\frac{(iq)^l}{(ip)^k}(1+\xi_{pq}+\eta_{pq}c_{pq})=a_{1N}+\rho_{pq}+(1+\varkappa_{pq})c_{pq}
\]
or
\[
c_{pq}\left(1+\varkappa_{pq}+\frac{(iq)^l}{(ip)^k}\eta_{pq}\right)
=-a_{1N}-\rho_{pq}-\frac{(iq)^l}{(ip)^k}(1+\xi_{pq}).
\]
Recall that condition \eqref{uslovie} imposed on the indices involves two
parameters to be chosen, namely, $\delta$ and $C$. In the case in question,
any choice of $\delta >0$ will fit. Fixing some $\delta$, by Lemma \ref{decay}
we choose $C$ so large that for all $p\geqslant C$
the coefficient of $c_{pq}$ in the last equation
becomes greater than $1/2$. Then the $c_{pq}$ will be uniformly bounded above.
Next, increasing $C$ further if necessary and again invoking Lemma \ref{decay},
we can ensure that the coefficient $1+\xi_{pq}+\eta_{pq}c_{pq}$ in \eqref{9}
be also greater than $1/2$. This guarantees the required conditions if $N=1$.

But if $N>1$, we plug the expression \eqref{9} for $\varphi_1$ in the first equation
among \eqref{82} to obtain
\[
-\partial_1^k\varphi_2=a_{11}z_1^pz_2^q+\frac{(iq)^l}{(ip)^k}(1+\xi_{pq}+\eta_{pq}c_{pq})z_1^pz_2^q,
\]
so that
\[
\varphi_2=-\frac{1}{(ip)^k}\left(a_{11}+\frac{(iq)^l}{(ip)^k}(1+\xi_{pq}+\eta_{pq}c_{pq})\right)z_1^pz_2^q.
\]
Continuing in the same manner, we reach the last equation among \eqref{82}, which yields (we
put $t= \frac{(iq)^l}{(ip)^k}$):
\[
\varphi_N=-\frac{1}{(ip)^k}(Q(t)+t^{N-1}(1+\xi_{pq}+\eta_{pq}c_{pq}))z_1^pz_2^q,
\]
where $Q$ is a polynomial of degree at most $N-2$ with coefficients depending on the
quantities $a_{1j}$. Then \eqref{83} implies the following equation for $c_{pq}$:
\[
-tQ(t)-t^N(1+\xi_{pq}+\eta_{pq}c_{pq})=a_{1N}+\rho_{pq}+(1+\varkappa_{pq})c_{pq}
\]
or
\[
c_{pq}(1+\varkappa_{pq}+t^N\eta_{pq})=-tQ(t)-t^N(1+\xi_{pq})-a_{1N}-\rho_{pq}.
\]
Now, choosing $\delta$ small and then $C$ large, we again ensure the uniform boundedness
of the $c_{pq}$ and also a uniform lower estimate for the coefficient in \eqref{9}.

\textbf{Case 2:} $j_0>0$, $j_1<N$. Also, for definiteness, we assume that
every open interval $(0,j_0)$, $(j_0,j_1)$, and $(j_1,N)$ contains a natural
number. Then system \eqref{syst} looks like this:
\begin{gather}
-\partial_1^k\varphi_1=z_1^pz_2^q(\xi_{pq}+\eta_{pq}c_{pq});\label{101}\\
\partial_2^l\varphi_j-\partial_1^k\varphi_{j+1}=0,\quad 0<j<j_0;\label{102}\\
\partial_2^l\varphi_{j_0}-\partial_1^k\varphi_{{j_0}+1}=z_1^pz_2^q;\label{103}\\
\partial_2^l\varphi_j-\partial_1^k\varphi_{j+1}=a_{1j}z_1^pz_2^q,\quad j_0<j<j_1;\label{104}\\
\partial_2^l\varphi_{j_1}-\partial_1^k\varphi_{{j_1}+1}=(a_{1j_1}+c_{pq})z_1^pz_2^q;\label{105}\\
\partial_2^l\varphi_j-\partial_1^k\varphi_{j+1}=(a_{1j}+a_{2j}c_{pq})z_1^pz_2^q,\quad j_1<j<N;\label{106}\\
\partial_2^l\varphi_N=(a_{1N}+a_{2N}c_{pq}+\rho_{pq}+\varkappa_{pq}c_{pq})z_1^pz_2^q.\label{107}
\end{gather}
Resolving equation \eqref{107}, we obtain
\[
\varphi_N=\frac{1}{(iq)^l}(a_{1N}+\rho_{pq}+(a_{2N}+\varkappa_{pq})c_{pq})z_1^pz_2^q.
\]
Then the last equation among those labeled by \eqref{106} takes the form
\begin{multline*}
\partial_2^l\varphi_{N-1}=\\
\left[\frac{(ip)^k}{(iq)^l}(a_{1N}+\rho_{pq}+(a_{2N}+\varkappa_{pq})c_{pq})+a_{1,N-1}+a_{2,N-1}c_{pq}\right]z_1^pz_2^q \\
=\left[\frac{(ip)^k}{(iq)^l}(\rho_{pq}+\varkappa_{pq}c_{pq})+a_{1,N-1}+a_{1N}\frac{(ip)^k}{(iq)^l}+(a_{2,N-1} + \frac{(ip)^k}{(iq)^l}a_{2N})c_{pq}\right]z_1^pz_2^q.
\end{multline*}
We continue to move ``upwards" until we obtain the following equation for
$\varphi_{j_1+1}$ from the first equation among the group \eqref{106} (we again put $t=(iq)^l/(ip)^k$):
\begin{equation}\label{11}
\partial_2^l\varphi_{j_1+1}=
\left[\left(\frac 1t\right)^{N-j_1-1}(\rho_{pq}+\varkappa_{pq}c_{pq})+A\left(\frac 1t\right)
+B\left(\frac 1t\right)c_{pq}\right]z_1^pz_2^q.
\end{equation}
Here $A$ and $B$ are certain polynomials of degree not exceeding $N$ and with
coefficients depending on the quantities $a_{1j},a_{2j}$ only.

Now we find an equation for $\varphi_{j_1}$, moving ``down" form \eqref{101}. We have
\[
\varphi_1=-\frac{1}{(ip)^k}(\xi_{pq}+\eta_{pq}c_{pq})z_1^pz_2^q.
\]
Next, solving the equations in the group \eqref{102} consecutively, we arrive at
\[\
\varphi_{j_0}=-\frac{1}{(ip)^k}t^{j_0-1}(\xi_{pq}+\eta_{pq}c_{pq})z_1^pz_2^q,
\]
after which \eqref{103} yields
\begin{equation}\label{e12}
\varphi_{j_0+1}=-\frac{1}{(ip)^k}\left[1+t^{j_0}(\xi_{pq}+\eta_{pq}c_{pq})\right]z_1^pz_2^q.
\end{equation}
(Recall that, eventually, we must also ensure that the factor in square brackets
in \eqref{e12} be bounded away from zero.) The next equation (which is the first in the
group \eqref{104}) then yields
\[
\varphi_{j_0+2}=-\frac{1}{(ip)^k}[a_{1,j_0+1}+t+t^{j_0+1}(\xi_{pq}+\eta_{pq}c_{pq})]z_1^pz_2^q.
\]
Continuing, we reach the last equation among \eqref{104} and obtain
\begin{equation}\label{e13}
\varphi_{j_1}=-\frac{1}{(ip)^k}\left[D(t)+t^{j_1-1}(\xi_{pq}+\eta_{pq}c_{pq})\right]z_1^pz_2^q,
\end{equation}
where, again, $D$ is a polynomial of degree not exceeding $N$ and with coefficients depending
on the $a_{1j}$.

Combining \eqref{11}, \eqref{e13}, and \eqref{105}, we obtain an equation for $c_{pq}$:
\begin{multline*}
-t\left[D(t)+t^{j_1-1}(\xi_{pq}+\eta_{pq}c_{pq})\right]\\
=\left(\frac 1t\right)^{N-j_1}(\rho_{pq}+\varkappa_{pq}c_{pq})+\frac 1t A\left(\frac 1t\right)
+\frac 1t B\left(\frac 1t\right)c_{pq} + a_{1j_1} + c_{pq}
\end{multline*}
or
\begin{multline*}
c_{pq}\left(1 + \frac 1t B\left(\frac 1t\right)+\left(\frac 1t\right)^{N-j_1}\varkappa_{pq}
+t^{j_1}\eta_{pq}\right)=\\
-tD(t)-t^{j_1}\xi_{pq}-\frac 1t A\left(\frac 1t\right)-\left(\frac 1t\right)^{N-j_1}\rho_{pq} - a_{1j_1}.
\end{multline*}

We must ensure that the coefficient of $c_{pq}$ on the left be bounded away from zero
uniformly in $p$ and $q$. Recall that we have imposed the condition $\delta/2\leqslant |t|\leqslant\delta$,
where $\delta >0$ is still to be chosen. We fix it so big that $|t^{-1}B(t^{-1})|\leqslant \frac{1}{4}$
for all $t$ with $|t|\geqslant\delta$. Than we invoke the restriction $p\geqslant C$ and, using Lemma
\ref{decay}, choose $C$ so large that
\[
\left(\frac{2}{\delta}\right)^{N-j_1}|\varkappa_{pq}|+\delta^{j_1}|\eta_{pq}|<\frac 14
\]
for all $p\geqslant C$. This ensures a uniform upper bound for the $c_{pq}$. Increasing $C$
further if necessary, we ensure also that the factor in square brackets in \eqref{e12}
be bounded away from zero. So, we are done.

To accomplish our task, it remains to establish the embedding theorem (Theorem \ref{main2}).
This will be done in the next section.

\section{Embedding theorems}\label{S3}

\subsection{On the plane}\label{plane}As has already been said, first we shall prove Theorem
\ref{main3} and then deduce Theorem \ref{main2} from it. 
\subsubsection{Several observations}
Convolving with an approximate identity, we may
replace the measures $\mu_j$ with infinitely differentiable compactly supported
functions $m_j$ and assume that all $\varphi_j$ are also compactly supported and
infinitely differentiable \textit{a priori}. For definiteness, we assume that $k$ is odd.

Like in \cite{KM1}, it suffices to prove the following statement. We recall the standard notation $\mathcal{D}$ 
for the space of infinitely differentable compactly supported functions on the plane.
\begin{theorem}\label{T3}
Let $\sigma, \tau \in \mathbb{C}$ be two nonzero complex numbers such that the numbers
$\tau_1=(2\pi i)^{l-k}\tau$ and $\sigma_1=(-1)^{l-k}(2\pi i)^{l-k}\overline{\sigma}$
are different and have nonsero imaginary parts of the same sign. Suppose that functions
$f,g,f_1,g_1 \in\mathcal{D}(\mathbb{R}^2)$ satisfy the relations
\begin{equation}\label{eq}
(\partial_1^k -\tau\partial_2^l)f_1 =f;\quad (\partial_1^k -\sigma\partial_2^l)g_1=g.
\end{equation} 
Then 
\begin{equation}\label{bilin}
\left|\langle f_1,g_1\rangle_{W_2^{\frac{k-1}{2},\frac{l-1}{2}}(\mathbb{R}^2)}\right|
\leqslant C_{\tau,\sigma}\left\|f\right\|_{L^1(\mathbb{R}^2)}\left\|g\right\|_{L^1(\mathbb{R}^2)}.
\end{equation}
\end{theorem}

By angular brackets, we have denoted the scalar product in
$W_2^{\frac{k-1}{2},\frac{l-1}{2}}(\mathbb{R}^2)$.
It should be noted that
\begin{equation}\label{norma}
\langle f_1,g_1\rangle_{W_2^{\frac{k-1}{2},\frac{l-1}{2}}(\mathbb{R}^2)}=
\int_{\mathbb{R}^2}\hat{f_1}(\xi,\eta)\overline{\hat{g_1}(\xi,\eta)}|\eta|^{l-1}\xi^{k-1}\,d \xi d \eta
\end{equation}
(we have lifted the modulus sign in $|\xi|^{k-1}$ because $k-1$ is even).
When $k=l=1$, a version of Theorem \ref{T3} with $\tau$ and $\sigma$ real was
proved in \cite{KM1}. In that case, both the statement and the proof were
much similar to the classical Gagliardo--Nirenberg inequality
$\left\|f\right\|_{L^2(\mathbb{R}^2)}\leqslant\left\|\partial_1 f\right\|_{L^1(\mathbb{R}^2)}
\left\|\partial_2 f\right\|_{L^1(\mathbb{R}^2)}$ valid for all smooth compactly
supported functions on the plane. In the case where $k\ne l$, the proof will be
different and harder. In fact, Theorem \ref{T3} remains true if
$\sigma_1$ and $\tau_1$ are real and have different signs, but the proof of this requires more calculations, so we content
ourselves with the case stated above. Now we show how to derive Theorem \ref{main3} from
Theorem \ref{T3}.

Suppose we have identities like \eqref{syst} with functions $m_j\in\mathcal{D}(\mathbb{R}^2)$
in place of the measures $\mu_j$, where the $\varphi_j$ also belong to $\mathcal{D}(\mathbb{R}^2)$.
We must prove inequality \eqref{est}, again with the $m_j$ in place of the $\mu_j$.
Taking a complex number $s$, consider the function $\psi_{s} = \sum\limits_{j=1}^{N}s^j\varphi_j$.
If we multiply the $j$th equation in \eqref{syst} by $s^j$ and then add all $N+1$ resulting equations,
we arrive at the identity $(\partial_1^k - s \partial_2^l)\psi_s = M_s$, where
$M_s = \sum\limits_{j=0}^{N}s^j m_j$. Writing another such identity with some number $t$ in place of
$s$, we obtain a system of the form \eqref{eq}.
The $L^1$-norms of the functions $M_s$ and $M_t$ do not exceed
$\max(1,|s|^N, |t|^N)\sum\limits_{j=0}^{N}\|m_j\|_{L^1}$.
So, if $s$ and $t$ satisfy the assumptions imposed on the parameters in Theorem \ref{T3}, we obtain
\[
\left|\langle\varphi_s,\varphi_t\rangle_{W_2^{\frac{k-1}{2},\frac{l-1}{2}}(\mathbb{R}^2)}\right|
\leqslant C_{s,t}\left(\sum\limits_{j=0}^{N}\|m_j\|_{L^1}\right)^2.
\]

Next, take $2n$ pairwise distinct numbers $\{s_j\}$ and $\{t_j\}$, $j=1,\dots,N$ for which
the above inequality holds true. 
We must now
estimate the quantities $a_{ij}=\langle \varphi_i,\varphi_j \rangle_{W_2^{\frac{k-1}{2},\frac{l-1}{2}}(\mathbb{R}^2)}$
(in fact, only for $i=j$, but it would be improper to impose this restriction at the moment).

Expanding, we obtain
\[
\langle \psi_{s_i},\psi_{t_j}\rangle_{W_2^{\frac{k-1}{2},\frac{l-1}{2}}(\mathbb{R}^2)}
= \sum\limits_{(p,q) \in [1,N]^2} (s_i)^{p} (t_j)^{q}a_{pq},
\]
where $i,j=1,\dots,N$. Thus, we have a system of $N^2$ linear equations for
the numbers $a_{pq}$. The matrix of this system is the tensor product of the matrices
$(s_i^{p})_{i,p}$ and $(t_j^{q})_{j,q}$; its determinant is the product of two
Vandermonde determinants of order $N$. Resolving this system, we
immediately arrive at the required estimate.
\subsubsection{The plot} Until the end of Subsection \ref{plane}, we shall
deal with Theorem \ref{T3}. Relations \eqref{eq} allow us to write out
an explicit integral formula for the scalar product to be estimated, by using
the Fourier transformation and its inverse. This formula
involves an improper integrals in two variables. Invoking residue
calculus, we reduce it to certain one-dimensional integrals.
Appropriate estimates for these integrals will finish the proof.

\subsubsection{Reduction to a two-dimensional improper integral}The functions $f_1$ and $g_1$
in \eqref{bilin} are infinitely differentiable and compactly supported \textit{a priori}.
Thus, their Fourier transforms decay rapidly at infinity, and the integral
on the right in \eqref{norma} exists in the usual Lebesgue sense. So, we can represent this integral
as a limit,
\[
\int\limits_{\mathbb{R}^2} =
\lim_{\substack{\varepsilon\rightarrow 0 \\ R \rightarrow \infty}}\,\int\limits_{\Omega_{\varepsilon,R}},
\]
where $\Omega_{\varepsilon,R}=\{\xi,\eta\in\mathbb{R}^2\colon \varepsilon\leqslant |\eta|\leqslant R\}$. 
Next, we replace the Fourier transforms in the integrand by their expressions
in terms of $f$ and $g$ found from the formulas
\[
((2\pi i\xi)^k - \tau (2\pi i\eta)^l)\hat{f_1}(\xi,\eta) =
\hat{f}(\xi,\eta); \quad ((2\pi i\xi)^k - \sigma (2\pi i\eta)^l)\hat{g_1}(\xi,\eta) = \hat{g}(\xi,\eta),
\]
which are the Fourier images of formulas \eqref{eq}. So, we must estimate the quantity
\[
\lim_{\substack{\varepsilon\rightarrow 0 \\ R \rightarrow \infty}}\,
\int\limits_{\Omega_{\varepsilon,R}}\frac{|\eta|^{l - 1}\xi^{k - 1}\hat{f}(\xi,\eta)\overline{\hat{g}(\xi,\eta)} d\xi d\eta}{((2\pi i\xi)^k -
\tau (2\pi i\eta)^l)((-2\pi i\xi)^k - \overline{\sigma} (-2\pi i\eta)^l)}.
\]
The denominator of the integrand does not vanish on $\mathbb{R}^2$ except at zero because of the assumptions about
$\sigma$ and $\tau$.
Then (for $\varepsilon$ and $R$ fixed) we replace $\hat{f}$ and $\hat{g}$ in the last
formula by their definitions in terms of $f$ and $g$, and change the order of integration:
\begin{equation}\label{quantity}
\lim_{\substack{\varepsilon \rightarrow 0 \\ R \rightarrow \infty}}
\iint\limits_{\substack{\supp f \\ \times\supp g}} F(\varepsilon,R,x_1,x_2,y_1,y_2)
f(x_1,x_2)\overline{g(y_1,y_2)} dy_1 dy_2 dx_1 dx_2,
\end{equation}
where
\begin{equation}\label{integrand}
F(\varepsilon,R,x_1,x_2,y_1,y_2)=\int\limits_{\Omega_{\varepsilon,R}}\frac{|\eta|^{l-1}\xi^{k-1}
e^{2\pi i ((x_1 - y_1)\xi + (x_2 - y_2)\eta)}}{((2\pi i\xi)^k - \tau (2\pi i\eta)^l)((-2\pi i\xi)^k -
\overline{\sigma} (-2\pi i\eta)^l)} d\xi d\eta.
\end{equation}
(The change of the order of integration is justified because the modulus of
the integrand in \eqref{integrand} is summable indeed over $\Omega_{\varepsilon,R}$ and does not
depend on $x_1$, $x_2$, $y_1$, and $y_2$.)

To prove Theorem \ref{T3}, we must show that the modulus of the quantity
\eqref{quantity} does not exceed $C\|f\|_1\|g\|_1$. For this, we prove that \textit{the function}
\eqref{integrand} \textit{is bounded uniformly in all $0<\varepsilon\leqslant R<\infty$ and almost all quads of reals
$x_1,\,x_2,\,y_1$, $y_2$}. The things become slightly more transparent if we integrate in
\eqref{integrand} first in $\xi$ and then in $\eta$, and in the inner integral in $\xi$ introduce
a new variable $\rho$ by setting $\xi=\rho |\eta|^{l/k}$. This yields the formula
\[
F(\varepsilon,R,x_1,x_2,y_1,y_2)=\frac{1}{(2\pi)^k}\int\limits_{\varepsilon\leqslant |\eta|\leqslant R}
|\eta|^{-1}\int\limits_{-\infty}^{+\infty}\frac{\rho^{k-1}e^{2\pi i(a\rho|\eta|^{l/k}+b\eta)}}{(\rho^k -
\tau_1(\sgn\eta)^l)(\rho^k -\sigma_1(\sgn\eta)^l)}\,d\rho d\eta,
\]
where we have put $a=x_1 - y_1$, $b=x_2-y_2$, and $\tau_1,\,\sigma_1$ were introduced in the statemant of
Theorem \ref{T3}.

Here integration in $\eta$ is over the union $[-R,-\varepsilon]\cup [\varepsilon,R]$, and we shall prove
the boundedness of the integral over each of these two intervals separately. Because of symmetry, we
consider only the interval $[\varepsilon,R]$. (Then $\sgn\eta$ in the denominator disappears. The argument below
will depend heavily on the condition that $\sigma_1$ and $\tau_1$ have imaginary parts of the same sign. When
treating the other interval, observe that multiplication of $\tau_1$ and $\sigma_1$ by $(-1)^l$ again yields
two numbers with the same property.) For definiteness,
we assume that $a>0$ (the opposite case reduces to this one if we change $\rho$ by $-\rho$, which again boils down to a possible
simultaneous change of sign of $\sigma_1$ and $\tau_1$; note that we may drop the case of $a=0$ because
this corresponds to a set of measure $0$). After that, we take
$(2\pi a)^{k/l}\eta$ for a new variable in the outer integral; this will modify the parameters $b$, $\varepsilon$, and $R$,
but will allow us to assume that $2\pi a=1$. So, finally, we must show that the following integral is bounded uniformly
in $0\leqslant\varepsilon <R$ and $b$:
\begin{equation}\label{ff}
\int\limits_{\varepsilon}^R \eta^{-1}\left(\int\limits_{-\infty}^{\infty}
\frac{\rho^{k-1}e^{i(\rho|\eta|^{l/k}+b\eta)}}{(\rho^k -\tau_1)(\rho^k -\sigma_1)}\right)\,d\rho\,d\eta.
\end{equation}
%
\subsubsection{Reduction to estimates for one-dimensional integrals}\label{s22}We shall calculate the
integral with respect to $\rho$ (i.e., the integral in parentheses in the above expression) with the
help of the residue formula, perceiving $\rho$ as a complex variable. Since the integrand
decays rapidly at infinity in the upper half-plane, integration over the countour that consists of the
interval $[-A, A]$ and the upper part of the circle of radius $A$ and centered at zero shows,
after the limit passage as $A\to\infty$, that the integral in question is $2\pi i$ times the
sum of the residues at the poles of the integrand in the upper half-plane. All these poles are simple and are $k$th
roots of $\tau_1$ or $\sigma_1$. Let $u^k=\tau_1$ (and $\Re u > 0$). Perceiving the integrand in question as
$\frac{\varphi(\rho)}{\psi(\rho)}$ with $\psi(\rho)=\rho^k -\tau_1$, we use the fact that $\varphi$ is regular at
$u$ to conclude that the residue at $u$ is
\[
\frac{\varphi(u)}{\psi^{\prime}(u)}=\frac{u^{k-1}e^{i(u|\eta|^{l/k}+b\eta)}}{ku^{k-1}(\tau_1 -\sigma_1)}
=\frac{e^{i(u|\eta|^{l/k}+b\eta)}}{k(\tau_1 -\sigma_1)}.
\]
Similarly, if $v^k=\sigma_1$ (and $\Re\sigma_1>0$), the residue at $v$ is equal to
\[
\frac{e^{i(v|\eta|^{l/k}+b\eta)}}{k(\sigma_1 -\tau_1)}.
\]
We shall need the following easy statement.
\begin{lem}If $z$ is nonreal and $k\in\mathbb{N}$, then the number of $k$th roots of $z$ in the
upper half-plane is $k/2$ if $k$ is even, $(k+1)/2$ if $k$ is odd and $\Im z>0$, and $(k-1)/2$ if
$k$ is odd and $\Im z<0$.
\end{lem}
\begin{proof}
We may assume that $|z|=1$. If $z$ lies in the upper half-plane, write it in the form $e^{it}$ with $0<t<\pi$. All $k$th
roots of $z$ are given by $e^{i(t+2\pi s)/k}$, $s=0,\dots,k-1$. Now,
\[
\pi\frac{2s}{k}<\frac{t+2\pi s}{k}<\pi\frac{2s+1}{k}.
\]
From this double inequality it easily follows that the biggest $s$ for which the argument $(t+2\pi s)/k$ is still smaller than
$\pi$ is $\frac{k}{2}-1$ if $k$ is even and $\frac{k-1}{2}$ if $k$ is odd, and we are done.

The case where $z$ lies in the lower half-plane is reduced to the preceding one by conjugation.
\end{proof}

Thus, under our assumptions the equations $\rho^k=\tau_1$ and $\rho^k=\sigma_1$ have one and the same number of
roots in the upper half plane. This shows that, up to a constant factor, the integral \eqref{ff} is equal
to a sum of $(k\pm 1)/2$ expressions of the form
\[
\frac{1}{k(u-v)}\int\limits_{\varepsilon}^R e^{ib\eta}\frac{e^{iu|\eta|^{l/k}}-e^{iv|\eta|^{l/k}}}{\eta}d\eta,
\]
where $u$ and $v$ are some $k$th roots of, respectively, $\tau_1$ and $\sigma_1$ in the upper half-plane.
(Note that, happily, the denominators in the above two formulas for residues are opposite to each other.)

\subsubsection{Estimates for one-dimensional integrals}\label{s23}
It is quite easy to realize that it suffices to estimate the above integrals
\begin{itemize}
\item with $R=1$, uniformly in $b$ and $0<\varepsilon <1$;
\item with $\varepsilon =1$, uniformly in $b$ and $R>1$.
\end{itemize}
We recall that $b$ is real, so we estimate the absolute value of the integrand by
$C|u-v|\eta^{\frac lk -1}$ in the first case and by
$e^{-\Im u |\eta|^{l/k}}+e^{-\Im v |\eta|^{l/k}}$ in the second. This finishes the proof
of Theorem \ref{T3} and, with it, of Theorem \ref{main3}.

\subsection{On the torus}In this subsection we prove Theorem \ref{main2}.
It should be noted that, when $k=l=1$, that theorem  was proved in \cite{KM1} \textit{by adjustment
of an argument for the plane to the periodic case}. In the present setting, the possibility of
such an adjustment
is questionable (it is not clear what should replace the residue theorem for
functions of discrete argument). So, we shall \textit{reduce} Theorem \ref{main2}
to its planar counterpart Theorem \ref{main3} rather than adjust the proof.
%

In the statement of Theorem \ref{main2}, by convolving with approximate identities,
we may assume that the measures $\mu_j$ and the functions $\varphi_j$ (see \eqref{syst})
are (proper) trigonometric polynomials. An immediate idea of reduction is
to perceive every function on $\mathbb{T}^2$ as a periodic function on the plane.
However, Theorem \ref{main3} applies to compactly supported functions, so that, apparently,
we must cut these periodic functions smoothly.
Eventually, this will work, but not in the most naive form.

We proceed to the details. First, we need yet another $L^1$-multiplier lemma.
\begin{lemma}\label{L1lemma}
Let $\tau$ be a complex number such that $i^{l-k}\tau$ is pure imaginary. Then
\begin{equation}\label{l1est}
\|\partial_1^{k-1}f_1\|_{L^1(\mathbb{T}^2)},\|\partial_2^{l-1}f_1\|_{L^1(\mathbb{T}^2)}
\leqslant C_{\tau}\|f\|_{L^1(\mathbb{T}^2)},
\end{equation}
where $f$ and $f_1$ are proper functions \textup{(}say, trigonometric polynomials for
definiteness\textup{)} related by the first equation in \eqref{eq}, i.e.,
$(\partial_1^k -\tau\partial_2^l)f_1 =f$.
\end{lemma}
\begin{proof}
We prove the lemma for $\partial_2^{l-1}f_1$, the other case is similar.
The Fourier coefficients of $f_1$ and $f$ are related as follows:
\begin{equation}\label{liason}
\hat{f_1}(m,n) = ((2\pi i m)^k - \tau (2\pi i n)^l)^{-1}\hat{f}(m,n).
\end{equation}
The function $((2\pi i m)^k - \tau (2\pi i n)^l)^{-1}$ is bounded on
$\mathbb{Z}^2\setminus\{(0,0)\}$ because we have assumed that $i^{k-l} \tau$ is pure imaginary.
We see that
$\partial_2^{l-1}f_1$ is obtained from $f$ by application of the multiplier with the following symbol:
\[
\frac{(2\pi in)^{l-1}}{((2\pi i m)^k - \tau (2\pi i n)^l)},\quad (m,n)\neq (0,0).
\]
This multiplier is bounded on $L^1(\mathbb{T}^2)$.
Again, this follows from Lemma \ref{L1}. See the hints to the proof of Lemma \ref{L2}.
\end{proof}

The above lemma yields a statement in the
spirit of Theorem \ref{main2} for the torus, yet in the $L^1$-metric.

\begin{lemma}\label{L2lemma}
Let proper trigonometric polynomials $\mu_j$, $j=0,\dots,N$, and $\varphi_j$, $j=1,\dots,N$, satisfy 
\eqref{syst}. Then
\[
\sum\limits_{j = 1}^{N}\|\partial_1^{k-1}\varphi_j\|_{L_1(\mathbb{T}^2)},\,\,
\sum\limits_{j = 1}^{N}\|\partial_2^{l-1}\varphi_j\|_{L_1(\mathbb{T}^2)}\leqslant
C\sum\limits_{j = 0}^{N}\|\mu_j\|_{L^1(\mathbb{T}^2)},
\]
where $C$ does not depend on the functions involved.
\end{lemma}
\begin{proof}
We argue as we did when deriving Theorem \ref{main3} from Theorem \ref{T3}.
Specifically, we construct the functions $\varphi_{s_v}$ in the same way as in that argument.
The numbers $s_v$ must be chosen so that the $i^{l-k}s_v$ be pure imaginary.
We apply Lemma \ref{L1lemma} to show that the $L^1$-norms of the functions
$\partial_1^{k-1}\varphi_{s_v}$ and $\partial_2^{l-1}\varphi_{s_v}$
are bounded in terms of the right-hand side of \eqref{est1}.
But we saw that the initial functions $\varphi_j$ are linear combinations of the $\varphi_{s_v}$
if we involve at least $N$ mutually distinct parameters $s_v$ (a system of linear equations
with a Vandermonde determinant arises). So, we are done.
\end{proof}

We need yet another fairly weak auxiliary estimate.
\begin{lemma}\label{L3lemma}
Let the assumptions of Lemma \textup{\ref{L2lemma}} be satisfied, and let $l>1$. Then
\[
\sum\limits_{j = 1}^{N}\|\varphi_j\|_{W_2^{\frac{k-1}{2}, 0}(\mathbb{T}^2)}
\leqslant C\sum\limits_{j = 0}^{N}\|\mu_j\|_{L^1(\mathbb{T}^2)}.
\]
\end{lemma}
\begin{proof}We proceed as in the above proof.
Specifically, we observe that, if the Fourier coefficients of two functions $f_1$ and $f$ are related by
formula \eqref{liason} and $f$ is integrable over the torus, then $f_1$ is square integrable
with an appropriate norm estimate, because the quantities $\frac{|m|^{(k-1)/2}}{(2\pi i m)^k - \tau (2\pi i n)^l}$
form a square summable sequence whenever $k,l\in\mathbb{N}$ and $k>1$. (It is still convenient to take $\tau$ in
such a way that $i^{l-k}\tau$ is pure imaginary.)
Next, we again invoke the functions $\varphi_{s_v}$ as in the proof of Lemma \ref{L2lemma}, etc.
\end{proof}
Surely, we can interchange the roles of $k$ and $l$, obtaining a similar statement for $l>1$.
Also, if $\max{k,l}>1$ and the $\varphi_j$ are proper, the above estimates imply
\[
\sum\limits_{j=1}^{N}\|\varphi_j\|_{L^2(\mathbb{T}^2)}
\leqslant C\sum\limits_{j=0}^{N}\|\mu_j\|_{L^1(\mathbb{T}^2)}.
\]
After these preparations, we can carry Theorem \ref{main3} over to the torus.
Consider a cut-off function $\chi$ on $\mathbb{R}^2$. We choose it nonnegative, finitely
supported, infinitely differentiable, and equal to $1$ on $[-2,2]\times [-2,2]$.
We define an operator $P$ on smooth functions on the torus by the formula 
$P(f) = \Tilde{f}\chi$, where $\Tilde{f}$ is the periodic extension of $f$ to the plane.
Let $a,b$ be nonnegative integers. Then $P$ maps the space $W^{a,b}_2(\mathbb{T}^2)$ continuously to
the space $\Tilde{W}^{a,b}_2(\mathbb{R}^2)$ whose norm is defined by the formula
$\left\|f\right\|_{\Tilde{W}^{a,b}_2(\mathbb{R}^2)} =
\left\|(1+|\xi|)^{a} (1+|\eta|)^{b}\hat{f}(\xi,\eta)\right\|_{L^2(\mathbb{R}^2)}$
(we have simply incorporated the junior derivatives in the Sobolev norm on the plane; recall that, for the
torus, the junior derivatives were involved in the norm from the outset).
Moreover, $P$ is an isomorphism onto its image, i.e., $\|f\|_{W^{a,b}_2(\mathbb{T}^2)}
\leqslant C\|Pf\|_{\Tilde{W}^{a,b}_2(\mathbb{R}^2)}$, because the norms in the spaces involved are
equivalent to the sum of the $L^2$-norms of the functions $\partial_1^u\partial_2^v f$ with
$u\leqslant a$, $v\leqslant b$ and $\chi$ is identically 1 on the unit square.

We want to show that $P$ possesses the same properties on Hilbert-type Sobolev spaces of
nonintegral order. So, let $a$ and $b$ be nonnegative reals. The fact that $P$ is still bounded
from $W^{a,b}_2(\mathbb{T}^2)$ to $\Tilde{W}^{a,b}_2(\mathbb{R}^2)$ is fairly easy. For example,
we can argue by complex interpolation. Alternatively, we can observe that
the Fourier transform of $Pf$ is the sum of shifts of $\hat{\chi}$ times Fourier coefficients of
$f$, and a direct estimate based on the rapid decay of $\hat{\chi}$ is possible.
The proof of the fact that $P$ is an isomorphism onto its image is more tricky.
We introduce an operator $Q$ from $\Tilde{W}^{a,b}_2(\mathbb{R}^2)$ to periodic functions
by putting
\begin{equation}
Qg(x,y) = \sum\limits_{m,n\in \mathbb{Z}}g(x + m, y + n)\chi(x + m,y + n).
\end{equation}
It can easily be seen that $Q$ takes $\Tilde{W}^{a,b}_2(\mathbb{R}^2)$ boundedly to
$W^{a,b}_2(\mathbb{T}^2)$ if $a$ and $b$ are nonnegative integers. By interpolation,
the same is true if $a,b$ are nonnegative reals. Now, we look at the operator $QP$, 
which acts on $W^{a,b}_2(\mathbb{T}^2)$. It is easily seen that $QP$ is multiplication
by the function $w(x,y)=\sum_{m,n} \chi^2(x+m,y+n)$. This function is strictly positive and
infinitely differentiable, hence $QP$ has bounded inverse (this inverse is multiplication by $1/w$).
It follows that $P$ is also an isomorphism onto its image.

Now, we can finish the proof of Theorem \ref{main2}. Recall that we assume all objects
involved in \eqref{syst} to be trigonometric polynomials. 
Consider the functions $P\varphi_j$ on the plane, then
\begin{multline*}
\sum\limits_{j=1}^{N}\|\varphi_j\|_{W_2^{\frac{k-1}{2},\frac{l-1}{2}}(\mathbb{T}^2)}
\leqslant C\sum\limits_{j=1}^{N}\|P\varphi_j\|_{\Tilde{W}_2^{\frac{k-1}{2},\frac{l-1}{2}}(\mathbb{R}^2)} \\
\leqslant C\sum\limits_{j=1}^{N}\|P\varphi_j\|_{W_2^{\frac{k-1}{2},\frac{l-1}{2}}(\mathbb{R}^2)}
+\quad\mbox{additional terms}.
\end{multline*}
If $k=l=1$, there are no additional terms. Otherwise, additional terms
of the form $\|P\varphi_j\|_{L^2(\mathbb{R}^2)}$ arise, along with terms of the form
$\|P\varphi_j\|_{W_2^{\frac{k-1}{2}, 0}(\mathbb{R}^2)}$ and $\|P\varphi_j\|_{W_2^{0, \frac{l-1}{2}}(\mathbb{R}^2)}$ if $\min k,l > 1$.
In any case, the additional terms are estimated
via the $L^1(\mathbb{T}^2)$-norms of the $\mu_j$ by Lemma \ref{L3lemma} and the
discussion after it.
To treat the remaining terms, we introduce the following functions $f_0,\dots,f_N$ on the plane:
\[
-\partial_1^k P\varphi_1=f_0;\quad\partial_2^l P\varphi_j- \partial_1^k P\varphi_{j+1} = f_j,
\quad j\in 1,\dots,N-1;\quad\partial_2^l P\varphi_N = f_N.
\]
By Theorem \ref{main3}, we have
\[
\sum\limits_{j = 1}^{N}\|P\varphi_j\|_{W_2^{(k-1)/2,(l-1)/2}(\mathbb{R}^2)}
\leqslant C\sum\limits_{j = 0}^{N}\|f_j\|_{L^1(\mathbb{R}^2)}.
\]
It remains to estimate the norms $\|f_j\|_{L^1(\mathbb{R}^2)}$ in terms of
the quantities $\|\mu_j\|_{L^1(\mathbb{T}^2)}$. By construction, we have
\begin{multline*}
f_j=\partial_2^l P\varphi_j-\partial_1^k P\varphi_{j+1}=
\partial_2^l (\chi\Tilde{\varphi}_{j}) - \partial_1^k (\chi\Tilde{\varphi}_{j+1}) = \\
\chi\Tilde{\mu_j}+\sum\limits_{q = 1}^{l}\binom lq\partial_2^q\chi\partial_2^{l-q}\Tilde{\varphi}_{j}-
\sum\limits_{q = 1}^{k}\binom kq\partial_1^q\chi\partial_1^{k-q}\Tilde{\varphi}_{j+1}
\end{multline*}
for $j=1,\dots,N-1$ (the formula differs slightly for $j=0,N$). Consequently, we have the inequality
\[
\|f_j\|_{L^1(\mathbb{R}^2)}
\leqslant C\|\mu_j\|_{L^1(\mathbb{T}^2)}+C\|\varphi_j\|_{W^{0,l-1}_1(\mathbb{T}^2)}+C\|\varphi_j\|_{W^{k-1,0}_1(\mathbb{T}^2)}.
\end{equation*}
Finally, by \ref{L1lemma}, we obtain
\[
\sum\limits_{j = 0}^{N}\|f_j\|_{L^1(\mathbb{R}^2)}
\leqslant C\sum\limits_{j = 0}^N\|\mu_j\|_{L^1(\mathbb{T}^2)}.
\]
This finishes the proof.

\end{document}